\documentclass[11pt,reqno]{amsart}

\usepackage{amsfonts,amssymb,amsmath,mathrsfs,graphicx}
\usepackage{float}
\usepackage{epsfig}
\usepackage{subcaption}
\usepackage{graphicx,colortbl,here}
\usepackage{soul}
\usepackage{tikz}

\usepackage[normalem]{ulem}

\usepackage{amsthm}

\usepackage{bm}

\usepackage{booktabs}

\usetikzlibrary{calc}
\usetikzlibrary{decorations.pathmorphing}
\allowdisplaybreaks
\makeatletter

\newcommand{\defhighlighter}[3][]{%
  \tikzset{every highlighter/.style={color=#2, fill opacity=#3, #1}}%
}

\defhighlighter{yellow}{.5}

\newcommand{\highlight@DoHighlight}{
  \fill [ decoration = {random steps, amplitude=1pt, segment length=15pt}
        , outer sep = -15pt, inner sep = 0pt, decorate
        , every highlighter, this highlighter ]
        ($(begin highlight)+(0,8pt)$) rectangle ($(end highlight)+(0,-3pt)$);
}

\newcommand{\highlight@BeginHighlight}{
  \coordinate (begin highlight) at (0,0);
}

\newcommand{\highlight@EndHighlight}{
  \coordinate (end highlight) at (0,0);
}

\newdimen\highlight@previous
\newdimen\highlight@current

\DeclareRobustCommand*\highlight[1][]{%
  \tikzset{this highlighter/.style={#1}}%
  \SOUL@setup
  \def\SOUL@preamble{%
      \highlight@BeginHighlight
      \highlight@EndHighlight
    \end{tikzpicture}%
  }%
  \def\SOUL@postamble{%
    \begin{tikzpicture}[overlay, remember picture]
      \highlight@EndHighlight
      \highlight@DoHighlight
    \end{tikzpicture}%
  }%
  \def\SOUL@everyhyphen{%
    \discretionary{%
      \SOUL@setkern\SOUL@hyphkern
      \SOUL@sethyphenchar
      \tikz[overlay, remember picture] \highlight@EndHighlight ;%
    }{%
    }{%
      \SOUL@setkern\SOUL@charkern
    }%
  }%
  \def\SOUL@everyexhyphen##1{%
    \SOUL@setkern\SOUL@hyphkern
    \hbox{##1}%
    \discretionary{%
      \tikz[overlay, remember picture] \highlight@EndHighlight ;%
    }{%
    }{%
      \SOUL@setkern\SOUL@charkern
    }%
  }%
  \def\SOUL@everysyllable{%
    \begin{tikzpicture}[overlay, remember picture]
      \path let \p0 = (begin highlight), \p1 = (0,0) in \pgfextra
        \global\highlight@previous=\y0
        \global\highlight@current =\y1
      \endpgfextra (0,0) ;
      \ifdim\highlight@current < \highlight@previous
        \highlight@DoHighlight
        \highlight@BeginHighlight
      \fi
    \end{tikzpicture}%
    \the\SOUL@syllable
    \tikz[overlay, remember picture] \highlight@EndHighlight ;%
  }%
  \SOUL@
}
\makeatother

\makeatletter
\renewenvironment{proof}[1][\proofname]{\par
  \pushQED{\qed}%
  \normalfont
  \topsep6\p@\@plus6\p@\relax
  \trivlist
  \item[\hskip\labelsep\bfseries #1\@addpunct{.}]\ignorespaces
}{%
  \popQED\endtrivlist\@endpefalse
}
\makeatother

\title[Dancing Equilibria in Asymmetric Kuramoto Networks]{
Existence and Stability of Dancing Equilibria in Asymmetric Kuramoto Networks
}

\author[Sun]{Wen Sun}
\address[Wen Sun]{\newline School of Mathematical Sciences,
\newline Dalian University of Technology, Dalian 116024, China}
\email{SunWen\_Vincent@163.com}

\author[Wang]{Yu-Qing Wang}
\address[Yu-Qing Wang]{\newline School of Mathematical Sciences,
\newline Dalian University of Technology, Dalian 116024, China}
\email{yqwang202309@163.com}

\author[Dong]{Jiu-Gang Dong$^*$}
\address[Jiu-Gang Dong]{\newline School of Mathematical Sciences,
\newline Dalian University of Technology, Dalian 116024, China}
\email{jgdong@dlut.edu.cn}



\thanks{\textbf{Acknowledgment.} This work was supported by the National Natural Science Foundation of China
through grant 12171069.}

\begin{document}

\newtheorem{theorem}{Theorem}[section]
\newtheorem{lemma}{Lemma}[section]
\newtheorem{corollary}{Corollary}[section]
\newtheorem{proposition}{Proposition}[section]
\newtheorem{remark}{Remark}[section]
\let\originalremark\remark
\let\endoriginalremark\endremark
\renewenvironment{remark}
  {\pushQED{\qed}\originalremark}
  {\popQED\endoriginalremark}

\newtheorem{definition}{Definition}[section]
\newtheorem{example}{Example}[section]

\let\originalexample\example
\let\endoriginalexample\endexample

\renewenvironment{example}
  {\pushQED{\qed}\originalexample}
  {\popQED\endoriginalexample}

\renewcommand{\theequation}{\thesection.\arabic{equation}}
\renewcommand{\thetheorem}{\thesection.\arabic{theorem}}
\renewcommand{\thelemma}{\thesection.\arabic{lemma}}
\newcommand{\bbr}{\mathbb R}
\newcommand{\bbz}{\mathbb Z}
\newcommand{\bbn}{\mathbb N}
\newcommand{\bbs}{\mathbb S}
\newcommand{\bbp}{\mathbb P}
\newcommand{\kp}{\kappa}
\newcommand{\mi}{\mathrm i}
\newcommand{\ddiv}{\textrm{div}}
\newcommand{\bn}{\bf n}
\newcommand{\rr}[1]{\rho_{{#1}}}
\newcommand{\thh}{\theta}
\def\charf {\mbox{{\text 1}\kern-.24em {\text l}}}
\renewcommand{\arraystretch}{1.5}

\begin{abstract}
  We study nonzero-frequency phase-locked motions in asymmetrically coupled Kuramoto networks.
  Such motions are relative equilibria with fixed phase differences and a nonzero common angular velocity, and we call them dancing equilibria.
  Their existence requires all coupling sums to have the same nonzero value.
  We show that neither symmetric coupling nor an acyclic associated digraph can support a dancing equilibrium.
  We introduce structurally equitable and $q$-twisted state equitable partitions and prove a partition-based criterion for the resulting class-constant profiles, with standard labeled $q$-twisted profiles recovered from singleton partitions.
  For the forward $m$-neighbor model, we characterize existence by an exact indivisibility criterion.
  Stability is studied modulo the common phase-shift direction.
  For general directed networks, strong connectivity and edgewise phase differences in $\left(-\pi/2,\pi/2\right)$ imply local orbital exponential stability and yield an explicit positively invariant set contained in the local basin of attraction.
  For arbitrary twisted indices, this contraction argument gives a low-winding stability regime with explicit positively invariant neighborhoods.
  For each existing $q$-twisted branch of the forward model, a discrete Fourier transform criterion yields local orbital exponential stability when all nonzero Fourier-mode factors are positive and nonlinear instability when at least one is negative.
  In the unstable case, the proof constructs explicit escaping real Fourier perturbations.
  We further derive additional explicit stability and instability ranges for arbitrary twisted indices in terms of constants $N$, $m$, and $q$.
  For the first two twisted branches, sharper arguments yield a first-mode transition criterion for $q=1$ and a complete finite-size classification for $q=2$, with the degenerate case in each branch handled separately.
\end{abstract}

\maketitle



\section{Introduction}\label{sec:introduction}
\setcounter{equation}{0}

Synchronization and phase locking in oscillator networks are central topics in nonlinear dynamics, with applications in physical, biological, chemical, and engineering systems \cite{A-B-D-S,D-B,St}.
Kuramoto-type models provide a versatile framework for studying the effects of natural-frequency heterogeneity, coupling strength, and network topology.
The classical Kuramoto model is widely used to study synchronization and phase locking \cite{Ku}.
Its finite-dimensional all-to-all form is
\begin{equation*}
  \dot{\theta}_i=\Omega_i+\frac{K}{N}\sum_{j=1}^{N}\sin\left(\theta_j-\theta_i\right), \quad i\in[N],
\end{equation*}
where $[N]:=\left\{1,2,\dots,N\right\}$, $\theta_i$ is the phase of oscillator $i$, $\Omega_i$ is its natural frequency, and $K>0$ is the coupling strength.
In the absence of coupling, oscillator $i$ rotates with angular velocity $\Omega_i$.
The coupling term depends only on pairwise phase differences and may lead to frequency synchronization or phase locking.
Ha and Ryoo studied finite-dimensional phase locking and critical coupling for the classical Kuramoto model, introduced the pathwise critical coupling strength, and derived explicit sufficient coupling bounds \cite{H-R}.
D\"orfler and Bullo provided a broader survey of synchronization in complex networks of phase oscillators \cite{D-B}.
To separate the effect of the weighted interaction network from natural-frequency heterogeneity, we assume that $\Omega_i=\Omega_0$ for every $i\in[N]$.
The rotating-frame transformation
\begin{equation*}
  \vartheta_i(t)=\theta_i(t)-\Omega_0 t,\quad i\in[N],
\end{equation*}
removes the common natural frequency without changing the phase differences.
After relabeling $\vartheta$ as $\theta$, we replace the uniform all-to-all coupling by a nonnegative weighted directed network:
\begin{equation}\label{eq:ackm}
  \dot{\theta}_i=\sum_{j=1}^{N}a_{ij}\sin\left(\theta_j-\theta_i\right),\quad i\in[N],
\end{equation}
where $A:=\left(a_{ij}\right)_{N\times N}$ and $a_{ij}\geq0$ is the coupling strength from oscillator $j$ to oscillator $i$.
Thus $a_{ij}>0$ means that oscillator $j$ acts on oscillator $i$, while the diagonal entries do not affect \eqref{eq:ackm}.
Weighted phase-oscillator models of this type appear in synchronization theory, power networks, biological systems, and multi-agent dynamics.
D\"orfler, Chertkov, and Bullo derived a synchronization condition for heterogeneous weighted oscillator networks in terms of the network topology, natural frequencies, and coupling strengths \cite{D-C-B}.
Jafarpour, Huang, Smith, and Bullo developed a framework for flow and elastic network equations on the $n$-torus, including vertex balance equations, winding structures, and network flows \cite{J-H-S-B}.

\vspace{0.5em}

The dynamics of \eqref{eq:ackm} depend strongly on the symmetry of $A$. When $A$ is symmetric, system \eqref{eq:ackm} is the gradient system
\begin{equation*}
  \dot{\theta}=-\nabla\mathcal V_A(\theta)
\end{equation*}
with the real-analytic potential
\begin{equation*}
  \mathcal V_A(\theta)=\frac{1}{2}\sum_{i,j=1}^{N}a_{ij}\left(1-\cos\left(\theta_j-\theta_i\right)\right).
\end{equation*}
Since $\mathcal V_A$ is real analytic on the compact phase space, the {\L}ojasiewicz gradient inequality implies convergence of every trajectory to an equilibrium in the rotating frame \cite{La,L-X}.
For the classical all-to-all model with identical natural frequencies, Dong and Xue proved frequency synchronization and convergence of phase differences for arbitrary initial data \cite{D-X}.
Local minimizers of $\mathcal V_A$ generate Lyapunov-stable equilibrium orbits. If the associated undirected graph is connected, the synchronized orbit is precisely the set of global minimizers.
Other critical points may be nonsynchronized local minima, saddle points, or degenerate equilibria.
The landscape of symmetric Kuramoto networks has been studied using nonconvex optimization, graph expansion, algebraic methods, and degeneracy analysis \cite{A-B-K-S-S-T,L-X-B,N-B-D-P-M-M,Scl}.
Symmetric finite-range ring networks have also been used to study synchronization, twisted states, and their basins of attraction.
Wiley, Strogatz, and Girvan investigated the basin sizes of coexisting twisted states in symmetrically coupled rings \cite{W-S-G}.
Zhang and Strogatz later showed that these basins can be octopus-like, with most of their volume contained in long tentacles rather than in the ball-like region near the attractor \cite{Z-S}.
For symmetric $A$, the coupling terms cancel in the summed locking equations. Therefore, nonzero-frequency phase locking can occur only under asymmetric coupling.

\vspace{0.5em}

For asymmetric coupling, neither the gradient structure nor the cancellation property persists in general.
Directed finite-range phase-oscillator networks may consequently exhibit collective behavior that is absent from their symmetric counterparts.
Jaros, Levchenko, Kapitaniak, and Maistrenko demonstrated chimera states and chaotic switching in unidirectionally coupled phase-oscillator networks with intermediate interaction ranges \cite{J-L-K-M}.
A directed network may also support a phase-locked configuration whose relative phases remain fixed while the entire configuration rotates with a nonzero common angular velocity.
A basic example is the nearest-neighbor model
\begin{equation}\label{eq:nearest_model}
  \dot{\theta}_i=K\sin\left(\theta_{\langle i+1\rangle}-\theta_i\right),\quad i\in[N],
\end{equation}
where $K>0$ and $\langle i+1\rangle$ is the unique element of $[N]$ congruent to $i+1$ modulo $N$.
System \eqref{eq:nearest_model} is the simplest directed circulant model in the present setting.
Rogge and Aeyels analyzed its phase-locked solutions and their stability for identical oscillators and also considered nonidentical natural frequencies and more general coupling functions \cite{R-A}.
Ha and Kang subsequently obtained explicit positive-measure subsets of the basins of the synchronized and splay states and proved nonlinear instability of the degenerate four-oscillator splay state \cite{H-K}.
We extend \eqref{eq:nearest_model} by allowing each oscillator to receive equal-strength input from its first $m$ forward neighbors. This gives the equal-strength forward $m$-neighbor model
\begin{equation}\label{eq:forward_m_model}
  \dot{\theta}_i=K\sum_{r=1}^{m}\sin\left(\theta_{\langle i+r\rangle}-\theta_i\right),\quad i\in[N],
\end{equation}
where $K>0$ and $m\in[N-1]$. For brevity, we refer to \eqref{eq:forward_m_model} as the forward $m$-neighbor model throughout the paper.
The case $m=1$ recovers the nearest-neighbor model studied in \cite{H-K,R-A}, while $m>1$ gives the directed finite-range extension analyzed below.
For comparison, the linear stability of $q$-states in symmetric finite-range rings has recently been studied within a broader framework for finite circulant Kuramoto networks \cite{Si-J-M-M-M-M-B}.

\vspace{0.5em}

Our goal is to characterize phase-locked motions with nonzero frequency in asymmetrically coupled networks and to determine their stability.
This shifts the focus from static configurations, in which every oscillator remains stationary, to coherent motions in which all phase differences remain fixed while the entire configuration rotates
with a nonzero common angular velocity. We refer to such motions as dancing equilibria.
We adopt the term dancing equilibrium from the sphere-valued multi-agent model introduced by Caponigro, Lai, and Piccoli for nonlinear opinion dynamics \cite{C-L-P}.
They studied static consensus and equilibrium configurations and showed that asymmetric interactions may also produce nonstationary motions in which all mutual distances remain fixed.
In representative examples, the agents undergo a rigid rotation on the sphere.
Piccoli subsequently discussed dancing equilibria as nontrivial asymptotic states of sphere-valued multi-agent systems and emphasized the role of interaction asymmetry \cite{Pi}.
In the present phase-oscillator setting, a dancing equilibrium is precisely a nonzero-frequency phase-locked motion.
Equivalently, it is a relative equilibrium generated by the common phase-shift symmetry, with common angular velocity $\Omega\neq0$.
This viewpoint leads to two complementary questions:
\begin{quote}
  \textbf{Existence.}
  \emph{Which directed weighted networks can support a configuration whose phase differences remain fixed while all oscillators rotate together with the same nonzero angular velocity?}
\end{quote}
\begin{quote}
  \textbf{Stability.}
  \emph{When such a rotating configuration exists, under what conditions do nearby phase configurations return to the same rotating pattern, up to a common shift of all phases?}
\end{quote}

\vspace{0.5em}

Our main results concern the existence and stability of nonzero-frequency locked profiles generated by asymmetric coupling.
We first show that a dancing equilibrium requires asymmetric coupling and a directed cycle in the associated digraph (Theorem \ref{thm:main_necessary_obstructions}).
We then introduce structural equitability and $q$-twisted state equitability and establish a partition-based criterion for class-constant profiles, with standard labeled $q$-twisted profiles recovered from singleton partitions (Theorem \ref{thm:main_equitable_partition_profiles}).
For the forward $m$-neighbor model, the existence of each prescribed $q$-twisted profile is characterized by an exact indivisibility criterion (Corollary \ref{cor:main_forward_q_existence}).
On the stability side, strong connectivity and edgewise phase differences in $\left(-\pi/2,\pi/2\right)$ imply local orbital exponential stability in quotient diameter and yield an explicit positively invariant neighborhood contained in the local basin of attraction (Theorem \ref{thm:main_edgewise_stability}).
For arbitrary twisted indices in the forward model, the corresponding specialization gives a low-winding regime in which the twisted orbit exists and is locally orbitally exponentially stable, with explicit positively invariant neighborhoods (Corollary \ref{cor:main_low_winding_stability}).
For any existing $q$-twisted branch, a discrete Fourier transform (DFT) criterion shows that positivity of every nonzero Fourier-mode factor implies local orbital exponential stability, while a negative factor implies nonlinear instability in the quotient sense (Theorem \ref{thm:main_general_q_fourier}).
The instability proof also constructs explicit escaping real Fourier perturbations.
For arbitrary twisted indices, we obtain the explicit stable range $N\geq\frac{\pi}{2}\left(2m+1\right)q^\sharp$ and, for existing branches, the instability range $N\leq\left(8m+3\right)\gcd\left(N,q\right)/3$, where $q^\sharp:=\min\left\{q,N-q\right\}$ and $\gcd\left(N,q\right)$ denotes the greatest common divisor of $N$ and $q$ (Corollary \ref{cor:main_general_q_parameter_ranges}).
For the fundamental twisted branch, the Fourier criterion yields explicit stable and unstable subranges and reduces the remaining transition window to the first-mode factor, with its unique zero case resolved separately (Corollary \ref{cor:main_q1_classification}).
For the $2$-twisted branch, it yields a complete finite-size classification, with stability for $N\geq6m+3$ and nonlinear instability in the quotient sense for $N\leq6m+2$ (Corollary \ref{cor:main_q2_classification}).
The unique fully degenerate case $\left(N,m\right)=\left(8,1\right)$ is resolved by a four-periodic reduction to the four-oscillator nearest-neighbor dynamics.

\vspace{0.5em}

The rest of the paper is organized as follows.
Section \ref{sec:preliminaries} introduces the notation and graph conventions, presents a three-oscillator example, and defines locked profiles, dancing equilibria, and quotient stability.
Section \ref{sec:main_results} states the main results.
Section \ref{sec:existence} proves the necessary structural conditions and constructive existence criteria.
Section \ref{sec:stability} develops the nonlinear stability arguments and the DFT-based spectral analysis.
Section \ref{sec:conclusion} summarizes the results and discusses further directions.
Appendix \ref{sec:technical_lemmas} collects the technical lemmas, while Appendix \ref{sec:model_lemma_proofs} contains the deferred proofs of the model-specific lemmas.


\section{Preliminaries}\label{sec:preliminaries}
\setcounter{equation}{0}

This section first fixes the notation and graph conventions and presents a three-oscillator example illustrating how asymmetric coupling can generate a common nonzero angular velocity.
It then introduces locked profiles, dancing equilibria, quotient stability, and the twisted profiles used throughout the paper.

\subsection{Notation, graph conventions, and a three-oscillator example}

Although system \eqref{eq:ackm} is well defined for smaller networks, the main results below concern systems with at least three oscillators.
We therefore assume $N\geq3$ throughout the paper.
For each positive integer $n$, we define $[n]:=\left\{1,2,\dots,n\right\}$.
For a $q$-twisted profile on $[N]$, we take $q\in[N-1]$, while for a twisted class pattern on a cyclically indexed partition with $p$ classes, we take $q\in[p-1]$.
In the forward $m$-neighbor model, we assume $K>0$ and $m\in[N-1]$.
Any additional restriction on $m$ will be stated explicitly when it is needed.
For $s\in\mathbb Z$, let $\langle s\rangle_n$ denote the unique representative in $[n]$ of the congruence class of $s$ modulo $n$.
When $n=N$, we omit the subscript and write $\langle s\rangle$.
The imaginary unit is denoted by $\mathrm i$.
For $x,y\in\mathbb Z$, we write $x\mid y$ if $y=kx$ for some $k\in\mathbb Z$, and write $x\nmid y$ otherwise.
For positive integers $x$ and $y$, let $\gcd(x,y)$ denote their greatest common divisor.
For $x\in\mathbb R$, let $\lfloor x\rfloor$ and $\lceil x\rceil$ denote the greatest integer not exceeding $x$ and the least integer not smaller than $x$, respectively.

\vspace{0.5em}

The digraph associated with the coupling matrix $A$ is
\begin{equation*}
  \mathcal G(A):=\left([N],\mathcal E_A\right),\quad
  \mathcal E_A:=\left\{\left(j,i\right)\in[N]\times[N]:i\neq j,\quad a_{ij}>0\right\}.
\end{equation*}
Thus $j\to i$ is an associated edge precisely when oscillator $j$ influences oscillator $i$ with positive coupling strength.
Self-loops are omitted because the diagonal entries of $A$ do not affect system \eqref{eq:ackm}.
A directed path from $i_0$ to $i_k$ is a finite sequence of distinct vertices $i_0,i_1,\dots,i_k$, where $k\geq0$, such that $\left(i_{r-1},i_r\right)\in\mathcal E_A$ for $1\leq r\leq k$.
A directed cycle is a finite sequence of distinct vertices $i_0,i_1,\dots,i_{k-1}$, where $k\geq2$, such that $\left(i_r,i_{r+1}\right)\in\mathcal E_A$ for $0\leq r\leq k-2$ and $\left(i_{k-1},i_0\right)\in\mathcal E_A$.
The digraph is acyclic if it contains no directed cycle.
It is strongly connected if, for every ordered pair of vertices, there exists a directed path from the first vertex to the second.

\vspace{0.5em}

The following example illustrates how asymmetric coupling can produce the same nonzero coupling sum at every oscillator.

\begin{example}
  Consider the three-oscillator system
  \begin{equation*}
  \left\{
    \begin{aligned}
      &\dot{\theta}_1=a_{12}\sin\left(\theta_2-\theta_1\right)+a_{13}\sin\left(\theta_3-\theta_1\right),\\
      &\dot{\theta}_2=a_{21}\sin\left(\theta_1-\theta_2\right)+a_{23}\sin\left(\theta_3-\theta_2\right),\\
      &\dot{\theta}_3=a_{31}\sin\left(\theta_1-\theta_3\right)+a_{32}\sin\left(\theta_2-\theta_3\right).
    \end{aligned}
  \right.
  \end{equation*}
  Set $\phi:=\left(0,\frac{2\pi}{3},\frac{4\pi}{3}\right)$ and assume that
  \begin{equation*}
    a_{12}-a_{13}=-a_{21}+a_{23}=a_{31}-a_{32}=:\tilde{a}\neq0,\quad \tilde{\Omega}:=\frac{\sqrt{3}}{2}\tilde{a}.
  \end{equation*}
  Direct substitution and the assumed relation give
  \begin{equation*}
    \sum_{j=1}^{3}a_{ij}\sin\left(\phi_j-\phi_i\right)
    =\frac{\sqrt{3}}{2}\tilde{a}
    =\tilde{\Omega},\quad i\in[3].
  \end{equation*}
  Consequently, the trajectory $\theta_i(t)=\phi_i+\tilde{\Omega} t$, $i\in[3]$, solves the system with common angular velocity $\tilde{\Omega}\neq0$ and constant phase differences.
\end{example}

\subsection{Locked profiles, dancing equilibria, and quotient stability}

The preceding example motivates our study of phase-locked motions with nonzero common frequency.
The phase space of system \eqref{eq:ackm} is $\mathbb T^N$, while local arguments are carried out after choosing a lift to $\mathbb R^N$.
Let $\mathbf 1:=\left(1,\dots,1\right)\in\mathbb R^N$.
Since the vector field depends only on phase differences, if $\theta(t)$ is a solution, then $\theta(t)+\gamma\mathbf 1$ is also a solution for every constant $\gamma\in\mathbb R$.
Hence the phase-shift direction is spanned by $\mathbf 1$.
A phase-locked motion is a solution for which all pairwise phase differences remain constant.
The next definition distinguishes the rotating motions considered here from stationary equilibria by requiring a nonzero common angular velocity.

\begin{definition}[Dancing equilibrium]
  A phase-locked solution $\theta(t)$ of system \eqref{eq:ackm} is called a dancing equilibrium if there exists a constant $\Omega\neq0$ such that
  \begin{equation*}
    \dot{\theta}_i(t)=\Omega,\quad i\in[N].
  \end{equation*}
\end{definition}

\begin{definition}[Locked profile]
  A pair $\left(\phi,\Omega\right)\in\mathbb R^N\times\mathbb R$ is called a locked profile if
  \begin{equation}\label{eq:locked_profile_equations}
    \sum_{j=1}^{N}a_{ij}\sin\left(\phi_j-\phi_i\right)=\Omega,\quad i\in[N].
  \end{equation}
  The value $\Omega$ is the common angular velocity associated with the profile.
  The locked profile is stationary when $\Omega=0$ and nonstationary when $\Omega\neq0$.
\end{definition}

\begin{lemma}\label{lem:locked_profile_equivalence}
  System \eqref{eq:ackm} admits a dancing equilibrium if and only if it admits a nonstationary locked profile.
  More precisely, if $\left(\phi,\Omega\right)$ satisfies \eqref{eq:locked_profile_equations} with $\Omega\neq0$, then
  \begin{equation*}
    \theta_i(t)=\phi_i+\Omega t,\quad i\in[N],
  \end{equation*}
  is a dancing equilibrium.
\end{lemma}

\begin{proof}
  If \eqref{eq:locked_profile_equations} holds with $\Omega\neq0$, direct substitution gives $\dot{\theta}_i(t)=\Omega$ for $i\in[N]$, and all phase differences are fixed.
  Hence the solution is a dancing equilibrium.
  Conversely, suppose that $\theta$ is a dancing equilibrium.
  By definition, there exists a constant $\Omega\neq0$ such that $\dot{\theta}_i(t)=\Omega$ for $i\in[N]$.
  Taking $\phi_i=\theta_i(0)$ for $i\in[N]$ and substituting $t=0$ into system \eqref{eq:ackm} yields \eqref{eq:locked_profile_equations}.
  Thus every dancing equilibrium determines a nonstationary locked profile after choosing real-valued phase representatives.
\end{proof}

A locked profile $\left(\phi,\Omega\right)$ determines the phase-locked motion
\begin{equation*}
  \theta(t)=\phi+\Omega t\mathbf 1.
\end{equation*}
Conversely, if all pairwise phase differences of a solution are constant, then each component of the vector field is constant and all phase derivatives are equal. Hence every phase-locked motion has this form.
It is a dancing equilibrium exactly when $\Omega\neq0$.

\begin{definition}[Dancing orbit]
  Let $\left(\phi,\Omega\right)$ be a nonstationary locked profile.
  The phase-shift orbit
  \begin{equation*}
    \mathcal O_\phi:=\left\{\phi+\gamma\mathbf 1:\gamma\in\mathbb R\right\}
  \end{equation*}
  is called the dancing orbit associated with $\left(\phi,\Omega\right)$.
\end{definition}

All elements of $\mathcal O_\phi$ differ only by a uniform phase shift.
Stability is therefore considered modulo $\operatorname{span}\left\{\mathbf 1\right\}$.
To measure perturbations on this quotient, define
\begin{equation*}
  D(u):=\max_i u_i-\min_i u_i.
\end{equation*}
Since $D\left(u+\gamma\mathbf 1\right)=D(u)$ and $D(u)=0$ if and only if $u\in\operatorname{span}\left\{\mathbf 1\right\}$,
the seminorm $D$ induces a norm on $\mathbb R^N/\operatorname{span}\left\{\mathbf 1\right\}$.

\begin{definition}[Quotient orbital exponential stability]\label{def:quotient_orbital_exponential_stability}
  Let $\left(\phi,\Omega\right)$ be a nonstationary locked profile.
  The orbit $\mathcal O_\phi$ is called locally orbitally exponentially stable in quotient diameter if there exist constants $\varepsilon>0$, $C\geq1$, and $\lambda>0$ such that,
  after choosing a local lift to $\mathbb R^N$, every solution $\theta(t)$ with
  \begin{equation*}
    D\left(\theta(0)-\phi\right)<\varepsilon
  \end{equation*}
  satisfies
  \begin{equation*}
    D\left(\theta(t)-\Omega t\mathbf 1-\phi\right)
    \leq C e^{-\lambda t}D\left(\theta(0)-\phi\right),\quad t\geq0.
  \end{equation*}
\end{definition}

We call $\mathcal O_\phi$ nonlinearly unstable in the quotient sense if there exists $\eta>0$ such that, for every sufficiently small $\varepsilon>0$, one can choose initial data in a local lift satisfying
\begin{equation*}
  D\left(\theta(0)-\phi\right)<\varepsilon
\end{equation*}
and a time $t_\varepsilon>0$ for which
\begin{equation*}
  D\left(\theta(t_\varepsilon)-\Omega t_\varepsilon\mathbf 1-\phi\right)\geq\eta.
\end{equation*}
The preceding definitions describe dancing equilibria and their quotient stability in general.
We next introduce a family of labeled phase profiles that will be used to construct and classify dancing equilibria in the forward model.
As recalled in the introduction, related $q$-twisted states are standard in the literature on symmetrically coupled oscillator rings and their basins of attraction.
Here the definition is tied to the fixed cyclic labeling and will be used for asymmetric nonzero-frequency locked motions.

\begin{definition}[$q$-twisted profile and $q$-twisted dancing equilibrium]\label{def:q_twisted_profile}
  For $q\in[N-1]$, define $\phi^{(q)}\in\mathbb R^N$ by
  \begin{equation}\label{eq:q_twisted_profile_definition}
    \phi_i^{(q)}:=\frac{2\pi q\left(i-1\right)}{N},\quad i\in[N].
  \end{equation}
  We call $\phi^{(q)}$ the $q$-twisted profile associated with the fixed cyclic labeling.
  This formula assigns phases according to the fixed vertex labels and is not a sorting convention.
  Its phase configuration on $\mathbb T^N$ is obtained componentwise modulo $2\pi$.
  If $\left(\phi^{(q)},\Omega\right)$ is a nonstationary locked profile, then
  \begin{equation*}
    \theta_i(t)=\phi_i^{(q)}+\Omega t,\quad i\in[N],
  \end{equation*}
  is called a $q$-twisted dancing equilibrium, and $\mathcal O_{\phi^{(q)}}$ is called the corresponding $q$-twisted dancing orbit.
\end{definition}


\section{Main results}\label{sec:main_results}
\setcounter{equation}{0}

For convenience, we first recall the two systems used throughout this section.
The general directed weighted Kuramoto system is
\begin{equation}\label{eq:ackm1}
  \dot{\theta}_i=\sum_{j=1}^{N}a_{ij}\sin\left(\theta_j-\theta_i\right),\quad i\in[N],
\end{equation}
and the forward $m$-neighbor model is
\begin{equation}\label{eq:forward_m_model1}
  \dot{\theta}_i=K\sum_{r=1}^{m}\sin\left(\theta_{\langle i+r\rangle}-\theta_i\right),\quad i\in[N].
\end{equation}
We now state the main results on the existence and stability of dancing equilibria.
For general directed networks, we give structural obstructions, a partition-based existence criterion, and a nonlinear orbital stability condition.
For the forward $m$-neighbor model, we derive an existence criterion and Fourier-based stability and instability results for prescribed twisted profiles.
Proofs of the existence and stability results are given in Section \ref{sec:existence} and Section \ref{sec:stability}, respectively.

\subsection{Existence results}

The existence results begin with structural obstructions to dancing equilibria.
We first show that symmetric coupling and acyclic associated digraphs exclude nonzero common angular velocities.
We then introduce a partition-based criterion for class-constant profiles generated by $q$-twisted state equitable partitions, with standard labeled $q$-twisted profiles recovered from singleton partitions.
Finally, for each prescribed $q$-twisted profile of the forward $m$-neighbor model, existence is characterized by an exact indivisibility criterion involving $N$, $m$, and $q$.

\vspace{0.5em}

The first theorem records the two necessary structural conditions.
Symmetry forces cancellation of the total coupling sum, while acyclicity forces at least one coupling sum to vanish.
Either property is incompatible with a common nonzero angular velocity.

\begin{theorem}\label{thm:main_necessary_obstructions}
  If system \eqref{eq:ackm1} admits a dancing equilibrium, then the coupling matrix $A$ is asymmetric and the associated digraph $\mathcal G(A)$ contains a directed cycle.
\end{theorem}

\begin{proof}
  The proof is given in Section \ref{sec:structural_obstructions}.
\end{proof}

We next introduce a class-level analogue of the twisted pattern from Definition \ref{def:q_twisted_profile}.

\begin{definition}[Structural equitability and $q$-twisted state equitability]\label{def:structural_twisted_equitability}
  Let
  \begin{equation*}
    \mathcal C:=\left\{C_1,C_2,\dots,C_p\right\}
  \end{equation*}
  be a partition of $[N]$ into $p$ nonempty classes, where $p\geq2$.
  The partition $\mathcal C$ is called structurally equitable for $A$ if, for every $r,s\in[p]$ with $r\neq s$,
  \begin{equation*}
    \sum_{j\in C_s}a_{ij}=\sum_{j\in C_s}a_{i'j},\quad i,i'\in C_r.
  \end{equation*}
  Let $q\in[p-1]$.
  A structurally equitable partition $\mathcal C$ is called $q$-twisted state equitable if there exists $\Omega_q\in\mathbb R$ such that, for every $r\in[p]$ and $i\in C_r$,
  \begin{equation*}
    \sum_{d=1}^{p-1}\sum_{j\in C_{\langle r+d\rangle_p}}a_{ij}\sin\left(\frac{2\pi qd}{p}\right)=\Omega_q.
  \end{equation*}
  For a $q$-twisted state equitable partition, define the associated class-constant profile $\alpha^{(q)}\in\mathbb R^N$ by
  \begin{equation*}
    \alpha_i^{(q)}:=\frac{2\pi q\left(r-1\right)}{p},\quad r\in[p],\quad i\in C_r.
  \end{equation*}
\end{definition}

The structural equitability condition means that vertices in the same class receive the same total input from every other class.
The following theorem shows that the associated class-constant profile yields a dancing equilibrium whenever $\Omega_q\neq0$.

\begin{theorem}\label{thm:main_equitable_partition_profiles}
  Let $\mathcal C$ be $q$-twisted state equitable for $A$.
  If $\Omega_q\neq0$, then $\left(\alpha^{(q)},\Omega_q\right)$ is a nonstationary locked profile of system \eqref{eq:ackm1}.
  Consequently,
  \begin{equation*}
    \theta_i(t)=\alpha_i^{(q)}+\Omega_q t,\quad i\in[N],
  \end{equation*}
  is a dancing equilibrium.
\end{theorem}

\begin{proof}
  The proof is given in Section \ref{sec:equitable_partition_profiles}.
\end{proof}

\begin{remark}
  {\rm
  The structural and $q$-twisted state equitability conditions play different roles. Structural equitability requires vertices in the same class to receive the same total input from every other class,
  whereas $q$-twisted state equitability additionally requires the coupling sums generated by the prescribed twisted class phases to agree across all classes.
  For example, let $p=3$, and suppose that every vertex in $C_r$ receives total weight $b_1$ from $C_{\langle r+1\rangle_3}$ and total weight $b_2$ from $C_{\langle r+2\rangle_3}$.
  Then the partition is structurally equitable, and the class phases $0$, $\frac{2\pi q}{3}$, and $\frac{4\pi q}{3}$ generate the common coupling sum
  \begin{equation*}
    \Omega_q=b_1\sin\left(\frac{2\pi q}{3}\right)+b_2\sin\left(\frac{4\pi q}{3}\right),\quad q\in[2].
  \end{equation*}
  In particular, $\Omega_1=\frac{\sqrt{3}}{2}\left(b_1-b_2\right)$, so unequal input weights in the two cyclic directions produce a nonzero common angular velocity, whereas $b_1=b_2$ gives a stationary class-constant profile.
  }
\end{remark}

For the singleton partition $C_r=\left\{r\right\}$, $r\in[N]$, structural equitability is automatic.
This partition is $q$-twisted state equitable precisely when there exists a common value $\Omega_q$ such that
\begin{equation*}
  \sum_{d=1}^{N-1}a_{i,\langle i+d\rangle}\sin\left(\frac{2\pi qd}{N}\right)=\Omega_q,\quad i\in[N].
\end{equation*}
In this case, the associated class-constant profile $\alpha^{(q)}$ is exactly the standard labeled $q$-twisted profile $\phi^{(q)}$ from Definition \ref{def:q_twisted_profile}.
The three-oscillator example in Section \ref{sec:preliminaries} is the case $N=p=3$ and $q=1$ of Theorem \ref{thm:main_equitable_partition_profiles}.
We now specialize this singleton case to the forward $m$-neighbor model.
For $q\in[N-1]$, set
\begin{equation*}
  \Omega_{N,m,q}:=K\sum_{r=1}^{m}\sin\left(\frac{2\pi q r}{N}\right),
\end{equation*}
which is the common coupling sum generated by $\phi^{(q)}$.

\begin{corollary}\label{cor:main_forward_q_existence}
  Consider the forward $m$-neighbor model \eqref{eq:forward_m_model1} and let $q\in[N-1]$.
  The pair $\left(\phi^{(q)},\Omega_{N,m,q}\right)$ is a nonstationary locked profile if and only if
  \begin{equation*}
    N\nmid mq \quad\text{and}\quad N\nmid\left(m+1\right)q.
  \end{equation*}
\end{corollary}

\begin{proof}
  The proof is given in Section \ref{sec:forward_existence}.
\end{proof}

\begin{remark}
  {\rm
  The angular velocity $\Omega_{N,m,q}$ is determined by the prescribed profile and the coupling pattern and is not an independent parameter.
  Because the finite sine sum can vanish for particular values of $m$ and $q$, existence is not monotone in the number $m$ of forward neighbors.
  At the endpoint $m=N-1$, the coupling is complete, equal-strength, and symmetric. Hence no dancing equilibrium exists.
  For the $1$-twisted branch, the criterion reduces to $m\in[N-2]$, and
  \begin{equation*}
    \Omega_{N,m,1}=K\frac{\sin\left(\frac{m\pi}{N}\right)\sin\left(\frac{\left(m+1\right)\pi}{N}\right)}{\sin\left(\frac{\pi}{N}\right)}.
  \end{equation*}
  When $m=1$, the model reduces to the nearest-neighbor model, and
  \begin{equation*}
    \Omega_{N,1,q}=K\sin\left(\frac{2\pi q}{N}\right).
  \end{equation*}
  }
\end{remark}

\subsection{Stability results}

The stability results combine a nonlinear argument based on phase differences along associated edges with an exact Fourier-mode analysis.
For general directed weighted networks, the nonlinear argument yields orbital exponential stability and an explicit positively invariant neighborhood.
For the forward $m$-neighbor model, this edgewise criterion first gives a low-winding stability regime with an explicit positively invariant neighborhood.
The Fourier analysis then gives an exact formula for the quotient spectrum and corresponding stability and instability criteria for prescribed twisted branches.
Its instability proof constructs explicit escaping perturbations along unstable real Fourier modes.
An additional corollary converts these criteria into explicit stable and unstable parameter ranges for arbitrary twisted indices.
We then specialize the Fourier analysis to the first two twisted branches.

\vspace{0.5em}

We begin with the result for a general directed weighted network.
Let $\left(\phi,\Omega\right)$ be a nonstationary locked profile.
For every $\left(j,i\right)\in\mathcal E_A$, let $d_{ij}\in\left(-\pi,\pi\right]$ be the unique representative of $(\phi_j-\phi_i)$ modulo $2\pi$.
These quantities describe phase differences along the associated edges and will be called the edgewise phase differences of the profile.
The relevant condition is $\left|d_{ij}\right|<\pi/2$ for every $\left(j,i\right)\in\mathcal E_A$.
When $\mathcal E_A\neq\varnothing$, define the corresponding margin by
\begin{equation*}
  \rho_\phi(A):=\min_{\left(j,i\right)\in\mathcal E_A}\left(\frac{\pi}{2}-\left|d_{ij}\right|\right).
\end{equation*}
For $0<\rho<\rho_\phi(A)$, define
\begin{equation*}
  \mathcal U_{\phi,\rho}:=\left\{\phi+\sigma\mathbf 1+u:\sigma\in\mathbb R,\quad u\in\mathbb R^N,\quad D(u)<\rho\right\}.
\end{equation*}

\begin{theorem}\label{thm:main_edgewise_stability}
  Let $\left(\phi,\Omega\right)$ be a nonstationary locked profile for system \eqref{eq:ackm1}.
  Assume that $\mathcal G(A)$ is strongly connected and $\rho_\phi(A)>0$.
  Then the dancing orbit $\mathcal O_\phi$ is locally orbitally exponentially stable in quotient diameter.
  More precisely, for every $0<\rho<\rho_\phi(A)$, the set $\mathcal U_{\phi,\rho}$ is positively invariant in a local rotating lift and is contained in the local basin of attraction of $\mathcal O_\phi$.
  Moreover, there exist constants $C_\rho\geq1$ and $\lambda_\rho>0$ such that every solution $\theta(t)$ with $\theta(0)\in\mathcal U_{\phi,\rho}$ satisfies
  \begin{equation*}
    D\left(\theta(t)-\Omega t\mathbf 1-\phi\right)\leq C_\rho e^{-\lambda_\rho t}D\left(\theta(0)-\phi\right),\quad t\geq0.
  \end{equation*}
\end{theorem}

\begin{proof}
  The proof is given in Section \ref{sec:edgewise_low_winding}.
\end{proof}

For the forward $m$-neighbor model, the edgewise phase-difference condition gives a low-winding stability regime with an explicit positively invariant neighborhood.
Introduce
\begin{equation*}
  q^\sharp:=\min\left\{q,N-q\right\}, \quad q\in[N-1].
\end{equation*}

\begin{corollary}\label{cor:main_low_winding_stability}
  Consider the $q$-twisted profile $\phi^{(q)}$ defined in \eqref{eq:q_twisted_profile_definition} for the forward $m$-neighbor model \eqref{eq:forward_m_model1}.
  If
  \begin{equation}\label{eq:q_twisted_low_winding_condition}
    4mq^\sharp<N,
  \end{equation}
  then the corresponding $q$-twisted dancing orbit exists and is locally orbitally exponentially stable.
  Moreover, for every $0<\rho<\rho_{\phi^{(q)}}(A)$, the set $\mathcal U_{\phi^{(q)},\rho}$ is positively invariant in a local rotating lift
  and is contained in the local basin of attraction of this orbit, where $\rho_{\phi^{(q)}}(A)=\frac{\pi}{2}-\frac{2\pi m q^\sharp}{N}>0$.
\end{corollary}

\begin{proof}
  The proof is given in Section \ref{sec:edgewise_low_winding}.
\end{proof}

We next use the circulant structure of the forward $m$-neighbor model to obtain an exact formula for the quotient spectrum and corresponding Fourier-mode criteria.
The edgewise phase-difference condition in Theorem \ref{thm:main_edgewise_stability} is sufficient but not necessary, whereas the circulant linearization of the forward model can be diagonalized explicitly by the DFT.
For each $\ell\in[N-1]$, define the corresponding mode factor
\begin{equation}\label{eq:gamma_general_q_definition}
  \Gamma_{N,m,q}^{(\ell)}:=\sum_{r=1}^{m}\cos\left(\frac{2\pi q r}{N}\right)\left(1-\cos\left(\frac{2\pi\ell r}{N}\right)\right), \quad \ell\in[N-1].
\end{equation}

\begin{theorem}\label{thm:main_general_q_fourier}
  Consider the $q$-twisted profile of the forward $m$-neighbor model \eqref{eq:forward_m_model1} and suppose that
  \begin{equation*}
    N\nmid mq\quad\text{and}\quad N\nmid\left(m+1\right)q.
  \end{equation*}
  Then the following conclusions hold.
  \begin{enumerate}
    \item If
    \begin{equation*}
      \Gamma_{N,m,q}^{(\ell)}>0,\quad \forall\ell\in[N-1],
    \end{equation*}
    then the corresponding $q$-twisted dancing orbit is locally orbitally exponentially stable.

    \item If there exists $\ell\in[N-1]$ such that
    \begin{equation*}
      \Gamma_{N,m,q}^{(\ell)}<0,
    \end{equation*}
    then the corresponding $q$-twisted dancing orbit is nonlinearly unstable in the quotient sense.
  \end{enumerate}
\end{theorem}

\begin{proof}
  The proof is given in Section \ref{sec:fourier_criterion}.
\end{proof}

\begin{remark}
  {\rm
  Furthermore, if $\Gamma_{N,m,q}^{(\ell_\ast)}<0$ for some $\ell_\ast\in[N-1]$, then the proof yields explicit escaping real Fourier perturbations.
  More precisely, there exist constants $\eta>0$ and $\varepsilon_0>0$ such that, for every $0<\left|\varepsilon\right|<\varepsilon_0$, the solution with initial data
  \begin{equation*}
    \theta_i(0)=\phi_i^{(q)}+\varepsilon\cos\left(\frac{2\pi\ell_\ast\left(i-1\right)}{N}\right),\quad i\in[N],
  \end{equation*}
  satisfies
  \begin{equation*}
    D\left(\theta(t_\varepsilon)-\Omega_{N,m,q}t_\varepsilon\mathbf 1-\phi^{(q)}\right)\geq\eta \quad  \text{for some } t_\varepsilon>0.
  \end{equation*}

  If all nonzero-mode factors are nonnegative and at least one vanishes, then the quotient linearization is nonhyperbolic and the theorem gives no stability conclusion.
  This alternative already occurs in the nearest-neighbor model.
  For $N=8$, $m=1$, and $q=2$, the existence conditions hold, but $\cos\left(2\pi q/N\right)=0$ and hence $\Gamma_{8,1,2}^{(\ell)}=0$ for every $\ell\in[7]$.
  Thus the entire quotient linearization is nonhyperbolic.
  This case is resolved by Lemma \ref{lem:q2_degenerate_instability} and is included in Corollary \ref{cor:main_q2_classification}.
  }
\end{remark}

The next corollary converts the Fourier criterion into explicit parameter ranges that apply to arbitrary twisted indices.

\begin{corollary}\label{cor:main_general_q_parameter_ranges}
  Consider the $q$-twisted profile of the forward $m$-neighbor model \eqref{eq:forward_m_model1}.
  Then the following conclusions hold.
  \begin{enumerate}
    \item If
    \begin{equation*}
      N\geq\frac{\pi}{2}\left(2m+1\right)q^\sharp,
    \end{equation*}
    then the corresponding $q$-twisted dancing orbit is locally orbitally exponentially stable.

    \item Suppose that $N\nmid mq$ and $N\nmid\left(m+1\right)q$.
    If
    \begin{equation*}
      N\leq\frac{8m+3}{3}\gcd(N,q),
    \end{equation*}
    then the corresponding $q$-twisted dancing orbit is nonlinearly unstable in the quotient sense.
  \end{enumerate}
\end{corollary}

\begin{proof}
  The proof is given in Section \ref{sec:fourier_criterion}.
\end{proof}

\begin{remark}
  {\rm
  No separate indivisibility assumption is needed in part (1).
  Indeed, set $g_q:=\gcd(N,q)$ and $N_q:=N/g_q$.
  Since $g_q$ divides both $q$ and $N-q$, the quotient $q^\sharp/g_q$ is a positive integer.
  The condition in part (1) therefore gives
  \begin{equation*}
    N_q\geq\frac{\pi}{2}\left(2m+1\right)\frac{q^\sharp}{g_q}
    \geq\frac{\pi}{2}\left(2m+1\right)>m+1.
  \end{equation*}
  Hence $N_q$ divides neither $m$ nor $m+1$.
  Since $\gcd\left(N_q,q/g_q\right)=1$, these two conclusions are equivalent to $N\nmid mq$ and $N\nmid\left(m+1\right)q$.

  Part (1) also includes the nearest-neighbor case $m=1$.
  In this case, the parameter condition gives $2\pi q^\sharp/N\leq4/3<\pi/2$, so the coefficient $\cos\left(2\pi q^\sharp/N\right)$ appearing in every mode factor is positive.
  For $m\geq2$, the stability conclusion follows from the uniform trigonometric positivity estimate in Lemma \ref{lem:uniform_trigonometric_positivity}.
  The condition in part (1) is equivalent to $2\pi q^\sharp/N\leq4/\left(2m+1\right)$, which explains the coefficient $\pi/2$.
  For $m=1$, the low-winding condition in Corollary \ref{cor:main_low_winding_stability} gives a larger stable range.
  For $m\geq2$, the stable range in part (1) is larger than the range supplied by the low-winding condition.
  The low-winding corollary nevertheless remains useful because it also provides explicit positively invariant neighborhoods.
  }
\end{remark}

We now specialize to the $1$-twisted profile. Corollary \ref{cor:main_forward_q_existence} shows that it defines a dancing equilibrium exactly when $m\in[N-2]$.
For this branch,
\begin{equation}\label{eq:gamma_N_m_1_ell_definition}
  \Gamma_{N,m,1}^{(\ell)}=\sum_{r=1}^{m}\cos\left(\frac{2\pi r}{N}\right)\left(1-\cos\left(\frac{2\pi\ell r}{N}\right)\right),\quad \ell\in[N-1].
\end{equation}
Comparisons between $N$ and $m$ yield explicit stable and unstable subranges, while the remaining transition window is governed by the first-mode factor.

\begin{corollary}\label{cor:main_q1_classification}
  Consider the $1$-twisted dancing orbit of the forward $m$-neighbor model \eqref{eq:forward_m_model1}, with $m\in[N-2]$.

  \begin{enumerate}
    \item If
    \begin{equation*}
      N\geq3m+2,
    \end{equation*}
    then this orbit is locally orbitally exponentially stable.

    \item If
    \begin{equation*}
      3\leq N\leq 2.9m,
    \end{equation*}
    then this orbit is nonlinearly unstable in the quotient sense.

    \item If
    \begin{equation*}
      2.9m<N\leq3m+1,
    \end{equation*}
    then the following first-mode alternatives hold:
    \begin{enumerate}
      \item if $\Gamma_{N,m,1}^{(1)}>0$, this orbit is locally orbitally exponentially stable.
      \item if $\Gamma_{N,m,1}^{(1)}<0$, this orbit is nonlinearly unstable in the quotient sense.
      \item if $\Gamma_{N,m,1}^{(1)}=0$, then $\left(N,m\right)=\left(4,1\right)$ and this orbit is nonlinearly unstable in the quotient sense.
    \end{enumerate}
  \end{enumerate}
\end{corollary}

\begin{proof}
  The proof is given in Section \ref{sec:q1_three_zone}.
\end{proof}

\begin{remark}
  {\rm
  When $m=1$, the condition in part (2) cannot hold because $N\geq3$.
  Part (1) covers every $N\geq5$, while part (3) contains the two remaining cases $N=3$ and $N=4$.
  For $N=3$, one has $\Gamma_{3,1,1}^{(1)}=-3/4<0$, so the negative alternative in part (3) gives nonlinear instability in the quotient sense.
  Lemma \ref{lem:q1_first_mode_zero_41} shows that $\left(N,m\right)=\left(4,1\right)$ is the unique zero case.
  At this pair, the edgewise margin is zero and every nonzero Fourier-mode factor vanishes, so the general edgewise and Fourier criteria are both inconclusive.
  The nonlinear instability of this four-oscillator nearest-neighbor orbit was proved in \cite[Sec. 4.1]{H-K}.

  The constant $2.9$ is an explicit finite-size threshold obtained from the estimates in the proof and is not intended to be sharp.
  If $m\to\infty$ and $N/m\to c\in\left(2,3\right)$, then
  \begin{equation*}
    \frac{4\Gamma_{N,m,1}^{(1)}}{2m+1}
    \longrightarrow
    \frac{\sin\left(2\pi/c\right)}{2\pi/c}
    \left(2-\cos\left(2\pi/c\right)\right)-1.
  \end{equation*}
  The limiting sign equation is $\sin\left(2\pi/c\right)\left(2-\cos\left(2\pi/c\right)\right)=2\pi/c$.
  With $x:=2\pi/c$, the left-hand side minus the right-hand side has derivative $2\cos x\left(1-\cos x\right)<0$ on $\left(2\pi/3,\pi\right)$, is positive at $x=2\pi/3$, and is negative at $x=\pi$.
  Hence the limiting equation has a unique solution $c_\ast\approx2.937$.
  Thus the rigorous threshold $2.9$ lies slightly below the limiting transition, while the remaining interval $2.9m<N\leq3m+1$ is stated in terms of the exact first-mode factor.
  The proof of the corollary relies only on the finite-size estimates.
  }
\end{remark}

We next specialize to the $2$-twisted branch.
By Corollary \ref{cor:main_forward_q_existence}, this branch exists precisely when
\begin{equation}\label{eq:q2_main_existence_conditions}
  N\nmid2m\quad\text{and}\quad N\nmid2\left(m+1\right).
\end{equation}
For the $2$-twisted branch, the stability transition is determined by the first Fourier mode.
The sign of this mode factor can be characterized explicitly in terms of $N$ and $m$, while finite trigonometric estimates control all remaining modes in the stable range.

\begin{corollary}\label{cor:main_q2_classification}
  Consider the $2$-twisted profile of the forward $m$-neighbor model \eqref{eq:forward_m_model1}, and suppose that \eqref{eq:q2_main_existence_conditions} holds.
  Then the corresponding $2$-twisted dancing orbit satisfies the following classification.

  \begin{enumerate}
    \item If
    \begin{equation*}
      N\geq6m+3,
    \end{equation*}
    then this orbit is locally orbitally exponentially stable.

    \item If
    \begin{equation*}
      N\leq6m+2,
    \end{equation*}
    then this orbit is nonlinearly unstable in the quotient sense.
  \end{enumerate}
\end{corollary}

\begin{proof}
  The proof is given in Section \ref{sec:q2_classification}.
\end{proof}

\begin{remark}
  {\rm
  The $1$-twisted and $2$-twisted classifications take different forms.
  For the $1$-twisted branch, the estimates above give explicit stable and unstable subranges, while the remaining transition range is governed by the exact first-mode factor.
  For the $2$-twisted branch, a branch-specific factorization and a triple-angle comparison yield the sharp threshold $N=6m+3$.
  Consequently, Corollary \ref{cor:main_q2_classification} is stated entirely through inequalities involving $N$ and $m$, whereas Corollary \ref{cor:main_q1_classification} retains a first-mode sign condition in the remaining transition range.

  Within the stable regime of the $2$-twisted branch, the available conclusions differ in strength.
  When $N\geq8m+1$, Corollary \ref{cor:main_low_winding_stability} provides explicit positively invariant neighborhoods contained in the local basin of attraction.
  For existing branches in the additional stable range $6m+3\leq N\leq8m$, positivity of all nonzero-mode factors gives local orbital exponential stability but does not by itself provide such a neighborhood.

  The pair $\left(N,m\right)=\left(8,1\right)$ is the unique fully degenerate case in the $2$-twisted classification.
  At this pair, every nonzero-mode factor vanishes, so Theorem \ref{thm:main_general_q_fourier} gives no conclusion.
  Lemma \ref{lem:q2_degenerate_instability} nevertheless proves nonlinear instability through a four-periodic reduction to the four-oscillator nearest-neighbor dynamics.
  For every other existing $2$-twisted branch in the unstable range, at least one Fourier-mode factor is negative.
  }
\end{remark}

\begin{remark}
  {\rm
  The mode factors and the existence conditions are invariant under the replacement $q\mapsto N-q$, while the common angular velocity changes sign.
  Hence the classifications for $q=1$ and $q=2$ also apply to the existing branches with twisted indices $N-1$ and $N-2$, respectively.

  Corollary \ref{cor:main_low_winding_stability} applies to every twisted index satisfying the low-winding condition and provides explicit positively invariant neighborhoods.
  Corollary \ref{cor:main_general_q_parameter_ranges} gives additional explicit stability and instability ranges for arbitrary twisted indices, although the specialized $q=1$ and $q=2$ classifications remain sharper.
  For $q=1$, part (2) of that corollary gives direct instability conclusions for the transition-range pairs $\left(N,m\right)=\left(6,2\right)$ and $\left(N,m\right)=\left(9,3\right)$.
  When $q^\sharp\geq3$, Corollary \ref{cor:main_general_q_parameter_ranges} still provides a sufficient finite-size range for arbitrary $q$. 
  This general range is typically less sharp than the specialized ranges obtained for $q^\sharp=1$ and $q^\sharp=2$, so a gap remains between the available sufficient condition and a complete finite-size classification. 
  Closing this gap would require sharper comparisons among the Fourier-mode factors.
  }
\end{remark}


\section{Proofs of existence results}\label{sec:existence}
\setcounter{equation}{0}

This section proves the existence results stated in Section \ref{sec:main_results}.
The argument has three parts.
We first prove that symmetry and acyclicity exclude dancing equilibria, then treat class-constant profiles associated with $q$-twisted state equitable partitions, and finally specialize this construction to the forward $m$-neighbor model.
By Lemma \ref{lem:locked_profile_equivalence}, a dancing equilibrium is equivalent to a nonstationary locked profile $\left(\phi,\Omega\right)$ satisfying
\begin{equation*}
  \Omega=\sum_{j\neq i}a_{ij}\sin\left(\phi_j-\phi_i\right),\quad i\in[N],
\end{equation*}
with $\Omega\neq0$.
Thus all coupling sums generated by the profile must have the same nonzero value.
The results below give general necessary conditions, a sufficient partition-based construction, and an exact criterion for prescribed $q$-twisted profiles in the forward model.
They do not classify all dancing equilibria for an arbitrary fixed coupling matrix.

\subsection{Structural obstructions}\label{sec:structural_obstructions}

We first prove the two necessary conditions in Theorem \ref{thm:main_necessary_obstructions}.
Symmetric coupling forces cancellation of the total coupling sum, whereas every finite directed acyclic graph contains a vertex with no incoming edge.

\begin{proof}[\textbf{Proof of Theorem \ref{thm:main_necessary_obstructions}}]
  Suppose that system \eqref{eq:ackm1} admits a dancing equilibrium. By Lemma \ref{lem:locked_profile_equivalence}, there exists a nonstationary locked profile $\left(\phi,\Omega\right)$ with $\Omega\neq0$.
  We first show that $A$ cannot be symmetric. If $A$ were symmetric, summing the locking equations over $i$ would yield
  \begin{equation*}
    N\Omega=\sum_{i,j=1}^{N}a_{ij}\sin\left(\phi_j-\phi_i\right).
  \end{equation*}
  For every pair of distinct indices $i$ and $j$, $a_{ij}\sin\left(\phi_j-\phi_i\right)+a_{ji}\sin\left(\phi_i-\phi_j\right)=0$ by the symmetry of $A$ and the oddness of the sine function.
  The diagonal terms vanish as well. Hence $N\Omega=0$, which contradicts $\Omega\neq0$. Thus $A$ is asymmetric.

  We next show that $\mathcal G(A)$ must contain a directed cycle. If $\mathcal G(A)$ were acyclic, then, as a finite directed acyclic graph, it would contain a vertex $i_\ast$ of indegree zero.
  By the definition of $\mathcal G(A)$, $a_{i_\ast j}=0$ for $j\neq i_\ast$.
  The locking equation for oscillator $i_\ast$ would then give
  \begin{equation*}
    \Omega=\sum_{j=1}^{N}a_{i_\ast j}\sin\left(\phi_j-\phi_{i_\ast}\right)=0,
  \end{equation*}
  again contradicting $\Omega\neq0$. Therefore $\mathcal G(A)$ contains a directed cycle.
\end{proof}

\subsection{Profiles from q-twisted state equitable partitions}\label{sec:equitable_partition_profiles}

The structural equitability condition is a weighted directed analogue of the external equitable partitions used in network theory, formulated here in terms of incoming weights \cite{G-R,S-B-M}.
It permits arbitrary coupling within each class and requires only that vertices in the same target class receive the same total input from every other source class.
The additional $q$-twisted state equitability condition requires the coupling sums generated by the prescribed twisted class phases to agree across all classes.

\begin{proof}[\textbf{Proof of Theorem \ref{thm:main_equitable_partition_profiles}}]
  Fix $r\in[p]$ and $i\in C_r$.
  If $j\in C_r$, then the definition of the class-constant profile gives $\alpha_j^{(q)}-\alpha_i^{(q)}\equiv0\pmod{2\pi}$, and hence the contribution from $C_r$ vanishes.
  If $d\in[p-1]$ and $j\in C_{\langle r+d\rangle_p}$, then $\alpha_j^{(q)}-\alpha_i^{(q)}\equiv\frac{2\pi qd}{p}\pmod{2\pi}$.
  Therefore
  \begin{equation*}
    \sum_{j=1}^{N}a_{ij}\sin\left(\alpha_j^{(q)}-\alpha_i^{(q)}\right)
    =\sum_{d=1}^{p-1}\sum_{j\in C_{\langle r+d\rangle_p}}a_{ij}\sin\left(\frac{2\pi qd}{p}\right)
    =\Omega_q,
  \end{equation*}
  where the last equality follows from $q$-twisted state equitability.
  Thus $\alpha^{(q)}$ satisfies the locking equations with common angular velocity $\Omega_q$.
  Since $\Omega_q\neq0$, Lemma \ref{lem:locked_profile_equivalence} gives the stated dancing equilibrium.
\end{proof}

\subsection{The forward m-neighbor model and q-twisted profiles}\label{sec:forward_existence}

We now specialize Theorem \ref{thm:main_equitable_partition_profiles} to the forward $m$-neighbor model
\begin{equation}\label{eq:forward_m_existence_model}
  \dot{\theta}_i=K\sum_{r=1}^{m}\sin\left(\theta_{\langle i+r\rangle}-\theta_i\right), \quad i\in[N].
\end{equation}
Equivalently, $a_{i,\langle i+r\rangle}=K$ for $r\in[m]$ and $i\in[N]$, and all other off-diagonal entries vanish.
The case $m=1$ is the nearest-neighbor model, while $m=N-1$ gives the complete graph with equal off-diagonal coupling weights.
For the prescribed profile \eqref{eq:q_twisted_profile_definition}, the common coupling sum in model \eqref{eq:forward_m_existence_model} is
\begin{equation*}
  \Omega_{N,m,q}=K\sum_{r=1}^{m}\sin\left(\frac{2\pi q r}{N}\right).
\end{equation*}
The standard finite trigonometric-sum identities needed to determine when this sum vanishes are stated in Lemma \ref{lem:finite_trig_sums} in Appendix \ref{sec:technical_lemmas}.

\begin{proof}[\textbf{Proof of Corollary \ref{cor:main_forward_q_existence}}]
  For the profile \eqref{eq:q_twisted_profile_definition}, one has
  \begin{equation*}
    \phi_{\langle i+r\rangle}^{(q)}-\phi_i^{(q)}
    \equiv\frac{2\pi q r}{N}\pmod{2\pi}, \quad r\in[m].
  \end{equation*}
  Therefore the $i$th coupling sum in the locking equations is
  \begin{equation*}
    K\sum_{r=1}^{m}\sin\left(\phi_{\langle i+r\rangle}^{(q)}-\phi_i^{(q)}\right)
    =K\sum_{r=1}^{m}\sin\left(\frac{2\pi q r}{N}\right)
    =\Omega_{N,m,q},
  \end{equation*}
  which is independent of $i$. Thus the locking equations hold with common angular velocity $\Omega_{N,m,q}$,
  and the profile gives a $q$-twisted dancing equilibrium exactly when $\Omega_{N,m,q}\neq0$.
  Applying Lemma \ref{lem:finite_trig_sums} with $x=2\pi q/N$ gives
  \begin{equation*}
    \Omega_{N,m,q}=K\frac{\sin\left(\frac{m\pi q}{N}\right)\sin\left(\frac{\left(m+1\right)\pi q}{N}\right)}{\sin\left(\frac{\pi q}{N}\right)}.
  \end{equation*}
  Since $q\in[N-1]$, the denominator is nonzero. Hence $\Omega_{N,m,q}\neq0$ if and only if neither $\sin\left(m\pi q/N\right)$ nor $\sin\left(\left(m+1\right)\pi q/N\right)$ vanishes.
  This is equivalent to
  \begin{equation*}
    N\nmid mq  \quad \text{and} \quad N\nmid \left(m+1\right)q.
  \end{equation*}
\end{proof}


\section{Proofs of stability results}\label{sec:stability}
\setcounter{equation}{0}

This section proves the stability results stated in Section \ref{sec:main_results}.
Since a dancing equilibrium is a relative equilibrium rather than an isolated equilibrium, we pass to a rotating frame and analyze the induced dynamics modulo $\operatorname{span}\left\{\mathbf 1\right\}$.
We first establish the Fourier representation and the rotating-frame quotient framework.
Next, we prove the edgewise stability criterion and its low-winding specialization.
We then prove the Fourier-mode criterion and the general parameter-range corollary, including the construction of explicit escaping perturbations.
The final two subsections treat the $1$-twisted and $2$-twisted classifications.
The technical tools used below are collected in Appendix \ref{sec:technical_lemmas}, while the proofs of the model-specific lemmas stated in this section are deferred to Appendix \ref{sec:model_lemma_proofs}.

\subsection{Fourier reduction and rotating-frame analysis}\label{sec:spectral_framework}

We first introduce the Fourier notation and derive the DFT reduction for the forward $m$-neighbor model.
We then derive the rotating-frame perturbation equation and the quotient spectral criterion.
The standard stability and DFT facts used below are collected in Appendix \ref{sec:technical_lemmas}.
The proofs of Lemma \ref{lem:forward_model_fourier_reduction} and Lemma \ref{lem:quotient_linearized_criterion} are deferred to Appendix \ref{sec:model_lemma_proofs}.
Under the positive-exponent convention, the vectors below are the columns of an unnormalized inverse DFT matrix.
Standard treatments of circulant matrices and their diagonalization by the DFT are given in \cite{Ba,Da,Gr,Su}, while \cite{O-S-S-P-P} presents an application-oriented account in vibration analysis.

\vspace{0.5em}

Let $\omega_N:=\exp\left(2\pi\mathrm i/N\right)$. For each $\ell\in\left\{0,1,\dots,N-1\right\}$, define the Fourier mode $v^{(\ell)}\in\mathbb C^N$ by
\begin{equation*}
  v_i^{(\ell)}:=\omega_N^{\ell\left(i-1\right)},\quad i\in[N].
\end{equation*}
For the linearization at the $q$-twisted branch, define
\begin{equation*}
  \left(L_{m,q}u\right)_i:=K\sum_{r=1}^{m}\cos\left(\frac{2\pi q r}{N}\right)\left(u_{\langle i+r\rangle}-u_i\right),\quad i\in[N].
\end{equation*}
The mode factors $\Gamma_{N,m,q}^{(\ell)}$ are defined in \eqref{eq:gamma_general_q_definition}.

\begin{lemma}\label{lem:forward_model_fourier_reduction}
  For each $\ell\in\left\{0,1,\dots,N-1\right\}$, the Fourier mode $v^{(\ell)}$ is an eigenvector of $L_{m,q}$ with eigenvalue
  \begin{equation*}
    \lambda_{\ell}^{(q)}
    =K\sum_{r=1}^{m}\cos\left(\frac{2\pi q r}{N}\right)
    \left(\omega_N^{\ell r}-1\right).
  \end{equation*}
  Moreover, $\operatorname{Re}\lambda_{\ell}^{(q)}=-K\Gamma_{N,m,q}^{(\ell)}$.
  The mode $\ell=0$ is the phase-shift mode.
  For every $\ell\in[N-1]$, the real subspace generated by the real and imaginary parts of $v^{(\ell)}$ is invariant under $L_{m,q}$ and transverse to $\operatorname{span}\left\{\mathbf 1\right\}$.
\end{lemma}

We next introduce rotating coordinates and the corresponding quotient formulation.
Define the vector field $F:\mathbb R^N\to\mathbb R^N$ componentwise by
\begin{equation*}
  F_i(\theta):=\sum_{j\neq i}a_{ij}\sin\left(\theta_j-\theta_i\right),\quad i\in[N].
\end{equation*}
Then system \eqref{eq:ackm1} is $\dot{\theta}=F(\theta)$.
Let $\left(\phi,\Omega\right)$ be a nonstationary locked profile with $\Omega\neq0$, so that $F(\phi)=\Omega\mathbf 1$.
The corresponding dancing equilibrium is the trajectory
$\theta^\ast(t)=\phi+\Omega t\mathbf 1$, and its phase-shift orbit is $\mathcal O_\phi$.
Set $u(t):=\theta(t)-\Omega t\mathbf 1-\phi$.
Using the phase-shift invariance of $F$, we find that $u=0$ is an equilibrium of
\begin{equation}\label{eq:general_perturbation_equation}
  \dot u=F(\phi+u)-F(\phi).
\end{equation}
The phase-shift symmetry corresponds to the subspace $\operatorname{span}\left\{\mathbf 1\right\}$.
Recall that $D(u)=\max_i u_i-\min_i u_i$.
Since $D$ is invariant under uniform phase shifts and vanishes exactly on $\operatorname{span}\left\{\mathbf 1\right\}$, it induces a norm on
$\mathcal Q:=\mathbb R^N/\operatorname{span}\left\{\mathbf 1\right\}$.
If $\Pi:\mathbb R^N\to\mathcal Q$ denotes the canonical projection, this norm is given by
$\left\|\Pi(u)\right\|_D:=D(u)$.
Consequently, $D(u(t))\to0$ means that the rotating-frame solution approaches the phase-shift orbit of the locked profile in quotient diameter.
All quotient-diameter neighborhoods below are understood in a chosen real-valued lift, or equivalently in a local chart on $\mathbb T^N$ modulo uniform phase shifts.
With this quotient formulation in place, the linearization of \eqref{eq:general_perturbation_equation} at $u=0$ is $\dot u=Ju$, where
\begin{equation*}
  J_{ij}=
  \begin{cases}
    a_{ij}\cos\left(\phi_j-\phi_i\right),
    & j\neq i, \\[4pt]
    -\displaystyle\sum_{k\neq i}
    a_{ik}\cos\left(\phi_k-\phi_i\right),
    & j=i.
  \end{cases}
\end{equation*}
Every row of $J$ sums to zero, and hence $J\mathbf 1=0$.
The zero eigenvalue is therefore forced by the continuous phase-shift symmetry rather than by an accidental degeneracy.
Indeed, every point of the form $u=\sigma\mathbf 1$ is an equilibrium of the rotating-frame system, which consequently induces an equilibrium on $\mathcal Q$.
The next lemma applies the finite-dimensional linearized stability principle to this quotient system.
A quotient spectrum contained in the open left half-plane yields orbital exponential stability, whereas a quotient eigenvalue with positive real part yields nonlinear instability in the quotient sense.

\begin{lemma}\label{lem:quotient_linearized_criterion}
  Assume that $0$ is a simple eigenvalue of $J$ in the algebraic sense and that every other eigenvalue of $J$ has strictly negative real part.
  Then the orbit $\mathcal O_\phi$ is locally orbitally exponentially stable in quotient diameter.
  If, instead, $J$ has an eigenvalue with strictly positive real part, then $\mathcal O_\phi$ is nonlinearly unstable in the quotient sense.
\end{lemma}

The local spectral criterion does not by itself provide an explicit positively invariant neighborhood or a nonlinear decay estimate.
To obtain both, we use the continuous-time consensus-contraction estimate stated as Lemma \ref{lem:uniform_consensus_contraction} in Appendix \ref{sec:technical_lemmas}.
Its hypotheses will be verified directly for the coefficients obtained from the mean-value representation of the perturbation equation.

\subsection{Edgewise and low-winding stability}\label{sec:edgewise_low_winding}

We first prove the general edgewise criterion in Theorem \ref{thm:main_edgewise_stability}.
Unlike the later Fourier arguments, this criterion is not restricted to equispaced or $q$-twisted profiles: it applies to any locked profile whose associated digraph is strongly connected and whose edgewise phase differences lie in $\left(-\pi/2,\pi/2\right)$.
The proof uses the edgewise phase differences $d_{ij}$, the margin $\rho_\phi(A)$, and the set $\mathcal U_{\phi,\rho}$ introduced in the theorem statement.
The decomposition $\phi+\sigma\mathbf 1+u$ of a point in $\mathcal U_{\phi,\rho}$ is not unique, since a constant vector may be transferred between $\sigma\mathbf 1$ and $u$.
This ambiguity is immaterial because the quotient diameter is invariant under uniform phase shifts.
Equivalently,
\begin{equation*}
  \mathcal U_{\phi,\rho}=\left\{\theta\in\mathbb R^N:D\left(\theta-\phi\right)<\rho\right\}.
\end{equation*}
On $\mathbb T^N$, the same notation denotes the projection of this local lifted set.

\begin{proof}[\textbf{Proof of Theorem \ref{thm:main_edgewise_stability}}]
  By Lemma \ref{lem:locked_profile_equivalence}, the locked profile $\left(\phi,\Omega\right)$ generates the dancing equilibrium $\theta_i^\ast(t)=\phi_i+\Omega t$ for $i\in[N]$.
  Write a nearby solution in rotating coordinates as $\theta_i(t)=\Omega t+\phi_i+u_i(t)$ for $i\in[N]$.
  By $2\pi$-periodicity, each phase difference in the trigonometric terms may be represented as $d_{ij}+u_j-u_i$.
  The mean-value formula gives the time-dependent consensus system
  \begin{equation}\label{eq:general_time_dependent_consensus_form}
    \dot u_i=\sum_{j\neq i}b_{ij}(t)\left(u_j-u_i\right),\quad i\in[N],
  \end{equation}
  where
  \begin{equation*}
    b_{ij}(t):=a_{ij}\int_0^1\cos\left(d_{ij}+s\left(u_j(t)-u_i(t)\right)\right)\,ds.
  \end{equation*}
  Since $u$ is continuously differentiable, every coefficient $b_{ij}$ is continuous.

  We first establish positive invariance.
  On any time interval on which $D(u(t))<\rho$, every associated edge and every $s\in[0,1]$ satisfy
  \begin{equation*}
    \left|d_{ij}+s\left(u_j(t)-u_i(t)\right)\right|
    \leq\left|d_{ij}\right|+D(u(t))
    <\frac{\pi}{2}.
  \end{equation*}
  Hence $b_{ij}(t)>0$ on every associated edge, while $b_{ij}(t)=0$ on every non-associated edge.
  The maximum principle for cooperative consensus systems implies that the componentwise maximum of $u(t)$ is nonincreasing and the componentwise minimum is nondecreasing.
  Thus $D(u(t))$ is nonincreasing as long as $D(u(t))<\rho$.
  By continuation, $D(u(0))<\rho$ implies
  \begin{equation*}
    D(u(t))\leq D(u(0))<\rho,\quad t\geq0,
  \end{equation*}
  so $\mathcal U_{\phi,\rho}$ is positively invariant.
  Define
  \begin{equation*}
    b_\ast(\rho):=\min_{\left(j,i\right)\in\mathcal E_A}a_{ij}\cos\left(\left|d_{ij}\right|+\rho\right)>0, \quad
    b^\ast:=\max_{\left(j,i\right)\in\mathcal E_A}a_{ij}.
  \end{equation*}
  Positive invariance gives $b_\ast(\rho)\leq b_{ij}(t)\leq b^\ast$ on every associated edge for all $t\geq0$, and the coefficient graph of \eqref{eq:general_time_dependent_consensus_form} is the strongly connected digraph $\mathcal G(A)$.
  Lemma \ref{lem:uniform_consensus_contraction} therefore yields constants $C_\rho\geq1$ and $\lambda_\rho>0$ such that
  \begin{equation*}
    D(u(t))\leq C_\rho e^{-\lambda_\rho t}D(u(0)),\quad t\geq0.
  \end{equation*}
  These constants depend only on the associated digraph, the weights on its edges, the locked profile, and the chosen radius $\rho$.

  The preceding estimate already gives exponential convergence in quotient diameter.
  To identify the limiting phase shift, note that the arithmetic mean need not be preserved because the system is directed.
  Nevertheless, the maximum component of $u(t)$ is nonincreasing and the minimum component is nondecreasing, so both extremal functions have limits.
  Since $D(u(t))$ tends to zero exponentially, these limits coincide at some $\sigma_\infty\in\mathbb R$.
  Hence every component of $u(t)$ converges to $\sigma_\infty$, and
  \begin{equation*}
    \left|u_i(t)-\sigma_\infty\right|\leq D(u(t)),\quad i\in[N].
  \end{equation*}
  Therefore
  \begin{equation*}
    \theta(t)-\Omega t\mathbf 1\longrightarrow\phi+\sigma_\infty\mathbf 1
  \end{equation*}
  exponentially.
  Thus $\mathcal U_{\phi,\rho}$ is contained in the local basin of attraction of $\mathcal O_\phi$, and the orbit is locally orbitally exponentially stable in quotient diameter.
\end{proof}

We next specialize the edgewise phase-difference stability result to standard labeled $q$-twisted profiles on directed weighted networks.
For every associated edge $\left(j,i\right)\in\mathcal E_A$, let $\delta_{ij}^{(q)}$ be the unique representative of $\phi_j^{(q)}-\phi_i^{(q)}$ in $\left(-\pi,\pi\right]$.
Thus
\begin{equation*}
  \delta_{ij}^{(q)}\equiv\frac{2\pi q\left(j-i\right)}{N}\pmod{2\pi}.
\end{equation*}
Whenever $\mathcal E_A\neq\varnothing$, define
\begin{equation*}
  \rho_q(A):=\min_{\left(j,i\right)\in\mathcal E_A}\left(\frac{\pi}{2}-\left|\delta_{ij}^{(q)}\right|\right).
\end{equation*}
For $0<\rho<\rho_q(A)$, define
\begin{equation*}
  \mathcal U_{q,\rho}:=\left\{\theta\in\mathbb R^N:D\left(\theta-\phi^{(q)}\right)<\rho\right\}.
\end{equation*}

\begin{corollary}\label{cor:q_twisted_edgewise_stability}
  Consider system \eqref{eq:ackm1} and let $q\in[N-1]$. Assume that there exists $\Omega\neq0$ such that
  \begin{equation*}
    \sum_{r=1}^{N-1}a_{i,\langle i+r\rangle}\sin\left(\frac{2\pi qr}{N}\right)=\Omega,\quad i\in[N].
  \end{equation*}
  Assume further that $\mathcal G(A)$ is strongly connected and that $\rho_q(A)>0$.
  Then the corresponding $q$-twisted dancing orbit is locally orbitally exponentially stable.
  Moreover, for every $0<\rho<\rho_q(A)$, the set $\mathcal U_{q,\rho}$ is positively invariant in a local rotating lift and is contained in the local basin of attraction of the $q$-twisted dancing orbit.
\end{corollary}

\begin{proof}
  For the singleton partition $C_i=\left\{i\right\}$, $i\in[N]$, structural equitability is automatic, while the displayed common coupling-sum condition is precisely the requirement that the partition be $q$-twisted state equitable.
  Theorem \ref{thm:main_equitable_partition_profiles} therefore gives the nonstationary locked profile $\left(\phi^{(q)},\Omega\right)$.
  The condition $\rho_q(A)>0$ places all edgewise phase differences in $\left(-\pi/2,\pi/2\right)$. For this profile,
  \begin{equation*}
    \rho_q(A)=\rho_{\phi^{(q)}}(A), \quad \mathcal U_{q,\rho}=\mathcal U_{\phi^{(q)},\rho}.
  \end{equation*}
  The conclusion follows from Theorem \ref{thm:main_edgewise_stability}.
\end{proof}

We now apply the edgewise criterion to the $q$-twisted profiles of the forward $m$-neighbor model and prove the low-winding stability corollary.
Recall that
\begin{equation*}
  q^\sharp=\min\left\{q,N-q\right\}.
\end{equation*}
This quantity is the smallest winding magnitude represented by $q$ modulo $N$ and determines the largest absolute edgewise phase difference generated by the forward interaction range.

\begin{proof}[\textbf{Proof of Corollary \ref{cor:main_low_winding_stability}}]
  Under condition \eqref{eq:q_twisted_low_winding_condition}, every forward shift $r\in[m]$ satisfies $0<2\pi q^\sharp r/N<\pi/2$.
  Since $m\geq1$, this condition also implies $q^\sharp<N/4$, and hence $q\neq N/2$.
  If $q<N/2$, then $q=q^\sharp$.
  If $q>N/2$, then $q\equiv-q^\sharp\pmod N$.
  Consequently, the edgewise phase differences are $2\pi q^\sharp r/N$, $r\in[m]$, or their negatives.
  Hence every edgewise phase difference lies in $\left(-\pi/2,\pi/2\right)$, and all terms in the common coupling sum have the same strict nonzero sign.
  Since the forward $m$-neighbor model is circulant, this coupling sum is the same at every oscillator.
  The profile therefore defines a $q$-twisted dancing equilibrium.
  The one-step edges form a directed cycle, so the associated digraph is strongly connected.
  The largest absolute edgewise phase difference is $2\pi m q^\sharp/N$, and therefore
  \begin{equation*}
    \rho_{\phi^{(q)}}(A)=\frac{\pi}{2}-\frac{2\pi m q^\sharp}{N}>0.
  \end{equation*}
  The conclusion follows from Corollary \ref{cor:q_twisted_edgewise_stability}.
\end{proof}

\subsection{Fourier criterion and general parameter ranges}\label{sec:fourier_criterion}

We first prove the Fourier-mode criterion.
The linearization at a $q$-twisted profile is circulant, so its quotient spectrum is determined by the nonzero DFT modes.

\begin{proof}[\textbf{Proof of Theorem \ref{thm:main_general_q_fourier}}]
  By Corollary \ref{cor:main_forward_q_existence}, the assumptions $N\nmid mq$ and $N\nmid\left(m+1\right)q$ imply that $\phi^{(q)}$ defines a $q$-twisted dancing equilibrium.
  In rotating coordinates, the linearized equation is $\dot u=L_{m,q}u$, where
  \begin{equation*}
    \left(L_{m,q}u\right)_i
    =K\sum_{r=1}^{m}\cos\left(\frac{2\pi q r}{N}\right)\left(u_{\langle i+r\rangle}-u_i\right).
  \end{equation*}
  By Lemma \ref{lem:forward_model_fourier_reduction}, the nonzero Fourier modes represent the quotient directions and $\operatorname{Re}\lambda_{\ell}^{(q)}=-K\Gamma_{N,m,q}^{(\ell)}$ for $\ell\in[N-1]$.
  If every nonzero-mode factor is positive, then every quotient eigenvalue has strictly negative real part, while the phase-shift mode $\ell=0$ is the only neutral mode.
  Lemma \ref{lem:quotient_linearized_criterion} therefore gives local orbital exponential stability.

  Suppose now that $\Gamma_{N,m,q}^{(\ell_\ast)}<0$ for some $\ell_\ast\in[N-1]$.
  The quotient eigenvalues on the associated real Fourier subspace have strictly positive real part, so Lemma \ref{lem:quotient_linearized_criterion} gives nonlinear instability in the quotient sense.
  We next establish the explicit escaping statement recorded in the remark following Theorem \ref{thm:main_general_q_fourier}.
  Define $w_{\ell_\ast}\in\mathbb R^N$ by $w_{\ell_\ast,i}:=\cos\left(2\pi\ell_\ast\left(i-1\right)/N\right)$ for $i\in[N]$.
  Lemma \ref{lem:finite_fourier_basis} gives $\sum_{i=1}^{N}w_{\ell_\ast,i}=0$, while $w_{\ell_\ast,1}=1$.
  Hence $w_{\ell_\ast}\notin\operatorname{span}\left\{\mathbf 1\right\}$ and $\widehat w_{\ell_\ast}:=\Pi\left(w_{\ell_\ast}\right)$ is a nonzero vector in the unstable real Fourier subspace.
  Let $E^{\mathrm{cs}}$ denote the center-stable spectral subspace of the quotient linearization.
  Then $\operatorname{span}\left\{\widehat w_{\ell_\ast}\right\}\cap E^{\mathrm{cs}}=\left\{0\right\}$.
  By the local center-stable manifold theorem \cite{H-P-S}, there exist a neighborhood $\mathcal V$ of the reduced equilibrium and a local center-stable manifold $\mathcal W_{\mathrm{loc}}^{\mathrm{cs}}$ such that every reduced solution remaining in $\mathcal V$ for all forward time starts in $\mathcal W_{\mathrm{loc}}^{\mathrm{cs}}$.
  Since this manifold is tangent to $E^{\mathrm{cs}}$, after shrinking $\mathcal V$ there exists $\varepsilon_1>0$ such that
  \begin{equation*}
    \varepsilon\widehat w_{\ell_\ast}\in\mathcal V\setminus\mathcal W_{\mathrm{loc}}^{\mathrm{cs}},\quad 0<\left|\varepsilon\right|<\varepsilon_1.
  \end{equation*}
  Choose $\eta>0$ such that the quotient ball of radius $\eta$ is contained in $\mathcal V$, and set
  \begin{equation*}
    \varepsilon_0:=\min\left\{\varepsilon_1,\frac{\eta}{2\left\|\widehat w_{\ell_\ast}\right\|_D}\right\}.
  \end{equation*}
  For $0<\left|\varepsilon\right|<\varepsilon_0$, take
  \begin{equation*}
    \theta_i(0)=\phi_i^{(q)}+\varepsilon\cos\left(\frac{2\pi\ell_\ast\left(i-1\right)}{N}\right),\quad i\in[N].
  \end{equation*}
  The corresponding quotient initial point is $\varepsilon\widehat w_{\ell_\ast}$, which lies inside the ball of radius $\eta$ but outside $\mathcal W_{\mathrm{loc}}^{\mathrm{cs}}$.
  The quotient trajectory therefore cannot remain in $\mathcal V$ for all forward time and hence cannot remain in this ball.
  By continuity, there exists $t_\varepsilon>0$ such that
  \begin{equation*}
    D\left(\theta(t_\varepsilon)-\Omega_{N,m,q}t_\varepsilon\mathbf 1-\phi^{(q)}\right)=\eta.
  \end{equation*}
  This establishes the explicit escaping statement recorded in the remark following Theorem \ref{thm:main_general_q_fourier}.
\end{proof}

We next establish the parameter estimates used in Corollary \ref{cor:main_general_q_parameter_ranges}.
The proof of the following lemma is deferred to Appendix \ref{sec:model_lemma_proofs}.

\begin{lemma}\label{lem:general_q_parameter_estimates}
  Let $q\in[N-1]$.
  \begin{enumerate}
    \item If $N\geq\frac{\pi}{2}\left(2m+1\right)q^\sharp$, then the existence conditions hold and $\Gamma_{N,m,q}^{(\ell)}>0$ for every $\ell\in[N-1]$.
    \item Suppose that $N\nmid mq$ and $N\nmid\left(m+1\right)q$. If $N\leq\frac{8m+3}{3}\gcd(N,q)$, then $\Gamma_{N,m,q}^{(q)}\leq-1/4$.
  \end{enumerate}
\end{lemma}

\begin{proof}[\textbf{Proof of Corollary \ref{cor:main_general_q_parameter_ranges}}]
  Under the condition in part (1), Lemma \ref{lem:general_q_parameter_estimates} gives the existence conditions and positivity of every nonzero-mode factor.
  Theorem \ref{thm:main_general_q_fourier} therefore gives local orbital exponential stability.
  Under the assumptions of part (2), the same lemma gives a negative factor in the mode $\ell=q$.
  The theorem then gives nonlinear instability in the quotient sense.
\end{proof}

\subsection{Refined classification of the 1-twisted branch}\label{sec:q1_three_zone}

We now specialize to the $1$-twisted branch. Corollary \ref{cor:main_forward_q_existence} shows that this branch exists precisely when $m\in[N-2]$.
Throughout the remainder of this subsection, we impose this condition and use the profile $\phi^{(1)}$ defined in \eqref{eq:q_twisted_profile_definition}.
Its common angular velocity is
\begin{equation*}
  \Omega_{N,m,1}=K\sum_{r=1}^{m}\sin\left(\frac{2\pi r}{N}\right)>0.
\end{equation*}
We first record the Fourier symbol of the quotient linearization at this profile. The linearized equation is
\begin{equation*}
  \dot u_i=K\sum_{r=1}^{m}\cos\left(\frac{2\pi r}{N}\right)\left(u_{\langle i+r\rangle}-u_i\right).
\end{equation*}
By Lemma \ref{lem:circulant_diagonalization}, the Fourier mode $v^{(\ell)}$ has eigenvalue
\begin{equation*}
  \lambda_{\ell}^{(1)}=K\sum_{r=1}^{m}\cos\left(\frac{2\pi r}{N}\right)\left(\omega_N^{\ell r}-1\right),\quad \ell=0,1,\dots,N-1.
\end{equation*}
The mode $\ell=0$ is the neutral phase-shift mode. For every nonzero mode,
\begin{equation*}
  \operatorname{Re}\lambda_{\ell}^{(1)}=-K\Gamma_{N,m,1}^{(\ell)},
\end{equation*}
where $\Gamma_{N,m,1}^{(\ell)}$ is given explicitly in \eqref{eq:gamma_N_m_1_ell_definition}.
Thus the quotient spectrum is described by explicit scalar trigonometric sums.
The first Fourier mode provides the sign estimates used below.
It satisfies
\begin{equation}\label{eq:q1_first_mode_gamma_closed_form}
  \Gamma_{N,m,1}^{(1)}
  =\frac{\sin\left(\frac{m\pi}{N}\right)\cos\left(\frac{\left(m+1\right)\pi}{N}\right)}{\sin\left(\frac{\pi}{N}\right)}
  -\frac{\sin\left(\frac{2m\pi}{N}\right)\cos\left(\frac{2\left(m+1\right)\pi}{N}\right)}{2\sin\left(\frac{2\pi}{N}\right)}-\frac{m}{2}.
\end{equation}
Indeed, set $x:=2\pi/N$. Then
\begin{equation*}
  \Gamma_{N,m,1}^{(1)}=\sum_{r=1}^{m}\cos(rx)-\sum_{r=1}^{m}\cos^2(rx).
\end{equation*}
By Lemma \ref{lem:finite_trig_sums},
\begin{equation*}
  \sum_{r=1}^{m}\cos(rx)=\frac{\sin(mx/2)\cos\left(\left(m+1\right)x/2\right)}{\sin(x/2)}.
\end{equation*}
Also,
\begin{equation*}
  \sum_{r=1}^{m}\cos^2(rx)
  =\frac{m}{2}+\frac{1}{2}\sum_{r=1}^{m}\cos(2rx)
  =\frac{m}{2}+\frac{\sin(mx)\cos\left(\left(m+1\right)x\right)}{2\sin x},
\end{equation*}
again by Lemma \ref{lem:finite_trig_sums}. Substituting $x=2\pi/N$ yields \eqref{eq:q1_first_mode_gamma_closed_form}.
Equivalently, if $n:=2m+1$, $y:=\pi/N$, and $\upsilon:=ny$, then
\begin{equation}\label{eq:q1_first_mode_gamma_closed_form_ny}
  4\Gamma_{N,m,1}^{(1)}
  =\frac{\sin\upsilon}{\sin y}\left(2-\frac{\cos\upsilon}{\cos y}\right)-n.
\end{equation}
The proofs of the four lemmas stated below are deferred to Appendix \ref{sec:model_lemma_proofs}.
We first extract two explicit sign ranges from the first-mode formula.

\begin{lemma}\label{lem:q1_first_mode_positive_range}
  Let $N\geq3$ and $m\in[N-2]$.
  If $N\geq3m+2$, then $\Gamma_{N,m,1}^{(1)}>0$.
\end{lemma}

\begin{lemma}\label{lem:q1_first_mode_negative_range}
  Let $m\geq2$.
  If $2m+2\leq N\leq2.9m$, then $\Gamma_{N,m,1}^{(1)}<0$.
\end{lemma}

The next two lemmas provide the key simplifications in the intermediate range.
The first shows that the mode factors are nonnegative for all $2\leq\ell\leq N-2$ and are strictly positive whenever the first-mode factor is positive.
The second identifies $\left(N,m\right)=\left(4,1\right)$ as the unique case in which the first-mode factor vanishes.
Thus, outside this exceptional case, the intermediate spectral alternatives are determined by the sign of $\Gamma_{N,m,1}^{(1)}$.

\begin{lemma}\label{lem:q1_intermediate_sign_reduction}
  Let $N\geq3$ and $m\in[N-2]$ satisfy $2m+2\leq N\leq4m$.
  Then, for every $2\leq\ell\leq N-2$,
  \begin{equation}\label{eq:q1_intermediate_complement_formula}
    \Gamma_{N,m,1}^{(\ell)}
    =\frac{1}{2}\sum_{r=m+1}^{N-m-1}\left(-\cos\left(\frac{2\pi r}{N}\right)\right)\left(1-\cos\left(\frac{2\pi\ell r}{N}\right)\right)
    \geq0.
  \end{equation}
  Moreover, if $\Gamma_{N,m,1}^{(1)}>0$, then $\Gamma_{N,m,1}^{(\ell)}>0$ for every $2\leq\ell\leq N-2$.
\end{lemma}

\begin{lemma}\label{lem:q1_first_mode_zero_41}
  Let $N\geq3$ and $m\in[N-2]$.
  Then $\Gamma_{N,m,1}^{(1)}=0$ if and only if $\left(N,m\right)=\left(4,1\right)$.
\end{lemma}

The preceding lemmas reduce the remaining stability question for the $1$-twisted branch to the first-mode factor, with its unique zero handled separately.
We now combine these estimates with the low-winding and DFT criteria.

\begin{proof}[\textbf{Proof of Corollary \ref{cor:main_q1_classification}}]
  Suppose first that $N\geq3m+2$.
  If $N\geq4m+1$, then $4m<N$, so Corollary \ref{cor:main_low_winding_stability}, with $q=1$, gives local orbital exponential stability together with the explicit positively invariant sets stated there.
  It remains to prove the stability assertion in the range $3m+2\leq N\leq4m$.
  Lemma \ref{lem:q1_first_mode_positive_range} gives $\Gamma_{N,m,1}^{(1)}>0$.
  Lemma \ref{lem:q1_intermediate_sign_reduction} gives $\Gamma_{N,m,1}^{(\ell)}>0$ for $2\leq\ell\leq N-2$.
  By symmetry, $\Gamma_{N,m,1}^{(N-1)}=\Gamma_{N,m,1}^{(1)}>0$.
  Hence $\Gamma_{N,m,1}^{(\ell)}>0$ for every $\ell\in[N-1]$, and Theorem \ref{thm:main_general_q_fourier} gives local orbital exponential stability.

  We next prove part (2).
  Suppose that $3\leq N\leq2.9m$.
  These inequalities imply $m\geq2$.
  If $N\leq2m+1$, then
  \begin{equation*}
    N\leq2m+1<\frac{8m+3}{3},
  \end{equation*}
  so Corollary \ref{cor:main_general_q_parameter_ranges} gives nonlinear instability.
  If instead $N\geq2m+2$, Lemma \ref{lem:q1_first_mode_negative_range} gives $\Gamma_{N,m,1}^{(1)}<0$, and Theorem \ref{thm:main_general_q_fourier} gives nonlinear instability in the quotient sense.

  Finally, suppose that $2.9m<N\leq3m+1$.
  If $m=1$, then $N\in\left\{3,4\right\}$.
  For $N=3$, one has
  \begin{equation*}
    N=3\leq\frac{11}{3}=\frac{8m+3}{3}\gcd(N,1),
  \end{equation*}
  so Corollary \ref{cor:main_general_q_parameter_ranges} gives nonlinear instability in the quotient sense.
  For $N=4$, Lemma \ref{lem:q1_first_mode_zero_41} gives $\Gamma_{4,1,1}^{(1)}=0$, and nonlinear instability follows from \cite[Sec. 4.1]{H-K}.

  It remains to consider $m\geq2$.
  Since $N$ is an integer, the inequalities $2.9m<N\leq3m+1$ imply $2m+2\leq N\leq4m$.
  If $\Gamma_{N,m,1}^{(1)}>0$, Lemma \ref{lem:q1_intermediate_sign_reduction} gives $\Gamma_{N,m,1}^{(\ell)}>0$ for $2\leq\ell\leq N-2$.
  By symmetry, $\Gamma_{N,m,1}^{(N-1)}=\Gamma_{N,m,1}^{(1)}>0$.
  Hence every nonzero-mode factor is positive, and Theorem \ref{thm:main_general_q_fourier} gives local orbital exponential stability.
  If $\Gamma_{N,m,1}^{(1)}<0$, the same theorem gives nonlinear instability in the quotient sense.
  The zero alternative cannot occur here, because Lemma \ref{lem:q1_first_mode_zero_41} would force $\left(N,m\right)=\left(4,1\right)$.
\end{proof}

\subsection{Complete classification for the 2-twisted branch}\label{sec:q2_classification}

We now specialize to $q=2$ and assume the existence conditions in \eqref{eq:q2_main_existence_conditions}.
As for the fundamental twisted branch, the first Fourier mode detects the transition.
Here, however, the sign of this mode factor can be characterized explicitly in terms of $N$ and $m$, while additional finite-sum estimates control all remaining modes on the stable side.
For this branch,
\begin{equation*}
  \Gamma_{N,m,2}^{(\ell)}
  =\sum_{r=1}^{m}\cos\left(\frac{4\pi r}{N}\right)
  \left(1-\cos\left(\frac{2\pi\ell r}{N}\right)\right),\quad \ell\in[N-1].
\end{equation*}
The proofs of the following three lemmas are deferred to Appendix \ref{sec:model_lemma_proofs}.

\begin{lemma}\label{lem:q2_positive_factors}
  Suppose that
  \begin{equation}\label{eq:q2_positive_range}
    m\geq2\quad\text{and}\quad N\geq6m+3.
  \end{equation}
  Then $\Gamma_{N,m,2}^{(\ell)}>0$ for every $\ell\in[N-1]$.
\end{lemma}

\begin{lemma}\label{lem:q2_first_mode_sign}
  Suppose that
  \begin{equation}\label{eq:q2_lower_intermediate_range}
    2m+2\leq N\leq6m+2.
  \end{equation}
  Then $\Gamma_{N,m,2}^{(1)}\leq0$. Equality holds if and only if $\left(N,m\right)=\left(8,1\right)$.
\end{lemma}

\begin{lemma}\label{lem:q2_degenerate_instability}
  For $\left(N,m,q\right)=\left(8,1,2\right)$,
  every nonzero-mode factor vanishes, but the corresponding $2$-twisted dancing orbit is nonlinearly unstable in the quotient sense.
\end{lemma}

The preceding lemmas settle the stable range, the first-mode sign in the complementary range, and the unique fully degenerate case.
We now combine these ingredients with the DFT-based criterion.

\begin{proof}[\textbf{Proof of Corollary \ref{cor:main_q2_classification}}]
  Suppose first that $N\geq6m+3$.
  If $m=1$, then $N\geq9$, and for every $\ell\in[N-1]$,
  \begin{equation*}
    \Gamma_{N,1,2}^{(\ell)}=\cos\left(\frac{4\pi}{N}\right)\left(1-\cos\left(\frac{2\pi\ell}{N}\right)\right)>0.
  \end{equation*}
  If $m\geq2$, Lemma \ref{lem:q2_positive_factors} gives the same conclusion.
  Thus every nonzero-mode factor is positive, and Theorem \ref{thm:main_general_q_fourier} gives local orbital exponential stability.

  Suppose next that $N\leq6m+2$.
  If $N\leq2m+1$, then
  \begin{equation*}
    N\leq2m+1<\frac{8m+3}{3}\leq\frac{8m+3}{3}\gcd(N,2),
  \end{equation*}
  so Corollary \ref{cor:main_general_q_parameter_ranges} gives nonlinear instability.
  It remains to consider $2m+2\leq N\leq6m+2$.
  Lemma \ref{lem:q2_first_mode_sign} gives $\Gamma_{N,m,2}^{(1)}<0$ unless $\left(N,m\right)=\left(8,1\right)$.
  For every other pair in this range, Theorem \ref{thm:main_general_q_fourier}, with $\ell=1$, gives nonlinear instability in the quotient sense.
  The remaining pair $\left(N,m\right)=\left(8,1\right)$ is nonlinearly unstable by Lemma \ref{lem:q2_degenerate_instability}.
\end{proof}


\section{Conclusion}\label{sec:conclusion}
\setcounter{equation}{0}

We developed an existence and quotient-stability framework for nonzero-frequency phase-locked motions in asymmetrically coupled Kuramoto networks.
The existence theory gives two necessary conditions: a dancing equilibrium requires asymmetric coupling and a directed cycle in the associated digraph.
It also introduces structural equitability and $q$-twisted state equitability and proves a partition-based criterion for class-constant profiles, including the standard labeled $q$-twisted profiles obtained from singleton partitions.
For the forward $m$-neighbor model, existence reduces to an exact indivisibility criterion.
The stability theory combines a contraction argument based on edgewise phase differences with a DFT-based criterion for the existing $q$-twisted branches of the forward model.
The contraction argument yields local orbital exponential stability together with an explicit positively invariant set contained in the local basin of attraction.
For arbitrary twisted indices, we obtain the explicit stable range $N\geq\frac{\pi}{2}\left(2m+1\right)q^\sharp$ and the instability range $N\leq\left(8m+3\right)\gcd(N,q)/3$ for existing branches.
For the $1$-twisted branch, the orbit is locally orbitally exponentially stable when $N\geq3m+2$ and nonlinearly unstable in the quotient sense when $3\leq N\leq2.9m$.
In the remaining transition window $2.9m<N\leq3m+1$, the sign of the first-mode factor determines the conclusion except at its unique zero $\left(N,m\right)=\left(4,1\right)$, where the orbit is also nonlinearly unstable.
For the $2$-twisted branch, the orbit is locally orbitally exponentially stable when $N\geq6m+3$ and nonlinearly unstable in the quotient sense when $N\leq6m+2$.
The fully degenerate pair $\left(N,m\right)=\left(8,1\right)$ is resolved by reduction to the four-oscillator nearest-neighbor model.
We also construct explicit escaping perturbations along unstable Fourier modes for arbitrary twisted indices satisfying the existence conditions.

\vspace{0.5em}

Several questions remain open.
The constructive existence criteria developed here identify broad classes of dancing equilibria, but they do not classify all nonstationary locked profiles supported by an arbitrary fixed asymmetric coupling matrix.
Moreover, the condition on the edgewise phase differences is sufficient rather than necessary, and sharper nonlinear criteria are needed to recover the additional stable regimes detected by the Fourier analysis of the forward $m$-neighbor model.
For branches with $q^\sharp\geq3$, Corollary \ref{cor:main_general_q_parameter_ranges} still provides explicit finite-size stability and instability ranges for arbitrary twisted indices. 
These general ranges are typically less sharp than the specialized classifications for $q^\sharp=1$ and $q^\sharp=2$, so the exact Fourier criterion must still be checked mode by mode for the remaining parameter values.
One possible direction is to establish finite low-mode reduction principles that guarantee positivity of all sufficiently high mode factors and reduce the stability problem to a finite set of $q$-dependent comparisons.
Another direction is to exploit periodic invariant reductions when the network size and twisted index satisfy compatible divisibility relations, thereby transferring instability information from lower-dimensional twisted branches to larger networks.
Finally, nonhyperbolic branches for which all nonzero-mode factors are nonnegative and at least one vanishes require a higher-order analysis, such as a center-manifold or normal-form reduction, of the corresponding critical modes.


\appendix

\section{Technical lemmas}\label{sec:technical_lemmas}
\setcounter{equation}{0}

This appendix collects the general tools from trigonometry, dynamical systems, Fourier analysis, and consensus theory used in the proofs.
Precise references are provided for the finite trigonometric identities, the linearized stability and instability principle, and the continuous-time consensus estimate, while the lemmas on the Fourier basis and circulant diagonalization are proved for completeness.
The final two lemmas provide the exact finite estimates and uniform trigonometric positivity needed for the general twisted-index parameter ranges.

\begin{lemma}\label{lem:finite_trig_sums}
  Let $m$ be a positive integer. If $x\not\equiv0\pmod{2\pi}$, then
  \begin{equation*}
    \sum_{r=1}^{m}\sin\left(rx\right)=\frac{\sin\left(mx/2\right)\sin\left(\left(m+1\right)x/2\right)}{\sin\left(x/2\right)}
  \end{equation*}
  and
  \begin{equation*}
    \sum_{r=1}^{m}\cos\left(rx\right)=\frac{\sin\left(mx/2\right)\cos\left(\left(m+1\right)x/2\right)}{\sin\left(x/2\right)}.
  \end{equation*}
\end{lemma}

These identities are obtained from \cite[Sec. 1.341, formulas 1 and 3]{G-Ry} by taking $n=m$ and setting the initial angle and common difference equal to $x$.

The following linearized stability and instability principle is standard. Relevant references include \cite{H-P-S,P-M}.

\begin{lemma}\label{lem:standard_linearized_criterion}
  Consider a smooth finite-dimensional autonomous system with an equilibrium $x_\ast$.
  If every eigenvalue of the linearization at $x_\ast$ has strictly negative real part, then $x_\ast$ is locally exponentially stable.
  If the linearization has an eigenvalue with strictly positive real part, then $x_\ast$ is nonlinearly unstable.
\end{lemma}

With the Fourier notation introduced in Section \ref{sec:spectral_framework}, the following standard facts will be used repeatedly.

\begin{lemma}\label{lem:finite_fourier_basis}
  The vectors $v^{(0)},v^{(1)},\dots,v^{(N-1)}$ form a basis of $\mathbb C^N$.
  If $\ell\not\equiv0\pmod N$, then
  \begin{equation*}
    \sum_{i=1}^{N}v_i^{(\ell)}=0.
  \end{equation*}
  Moreover, $v_{\langle i+r\rangle}^{(\ell)}=\omega_N^{\ell r}v_i^{(\ell)}$ for $i\in[N]$ and $r\in\mathbb Z$.
\end{lemma}

\begin{proof}
  For $k,\ell\in\left\{0,1,\dots,N-1\right\}$, the standard Hermitian inner product gives
  \begin{equation*}
    \left\langle v^{(k)},v^{(\ell)}\right\rangle
    =\sum_{j=0}^{N-1}\overline{\omega_N^{kj}}\omega_N^{\ell j}
    =\sum_{j=0}^{N-1}\omega_N^{\left(\ell-k\right)j}.
  \end{equation*}
  If $k=\ell$, this sum equals $N$.
  If $k\neq\ell$, it is a finite geometric sum with ratio not equal to $1$, and therefore
  \begin{equation*}
    \sum_{j=0}^{N-1}\omega_N^{\left(\ell-k\right)j}
    =\frac{1-\omega_N^{\left(\ell-k\right)N}}{1-\omega_N^{\ell-k}}
    =0.
  \end{equation*}
  Thus the $N$ Fourier modes are nonzero and mutually orthogonal, so they form a basis of $\mathbb C^N$.
  Taking $k=0$ gives $\sum_{i=1}^{N}v_i^{(\ell)}=0$ whenever $\ell\not\equiv0\pmod N$.
  Finally, since $\langle i+r\rangle-1\equiv i+r-1\pmod N$ and $\omega_N^{\ell N}=1$, one has
  \begin{equation*}
    v_{\langle i+r\rangle}^{(\ell)}
    =\omega_N^{\ell\left(\langle i+r\rangle-1\right)}
    =\omega_N^{\ell\left(i+r-1\right)}
    =\omega_N^{\ell r}v_i^{(\ell)}.
  \end{equation*}
\end{proof}

For real coefficients $c_1,\dots,c_m$, define the circulant difference operator $L$ by
\begin{equation*}
  \left(Lu\right)_i:=\sum_{r=1}^{m}c_r\left(u_{\langle i+r\rangle}-u_i\right),\quad i\in[N].
\end{equation*}

\begin{lemma}\label{lem:circulant_diagonalization}
  Each Fourier mode $v^{(\ell)}$ is an eigenvector of $L$, and
  \begin{equation*}
    Lv^{(\ell)}=\lambda_\ell v^{(\ell)},\quad
    \lambda_\ell=\sum_{r=1}^{m}c_r\left(\omega_N^{\ell r}-1\right).
  \end{equation*}
  Consequently,
  \begin{equation*}
    \operatorname{Re}\lambda_\ell=-\sum_{r=1}^{m}c_r\left(1-\cos\left(\frac{2\pi\ell r}{N}\right)\right).
  \end{equation*}
\end{lemma}

\begin{proof}
  By Lemma \ref{lem:finite_fourier_basis}, we have $v_{\langle i+r\rangle}^{(\ell)}=\omega_N^{\ell r}v_i^{(\ell)}$.
  Therefore
  \begin{equation*}
    \left(Lv^{(\ell)}\right)_i
    =\sum_{r=1}^{m}c_r\left(v_{\langle i+r\rangle}^{(\ell)}-v_i^{(\ell)}\right)
    =\sum_{r=1}^{m}c_r\left(\omega_N^{\ell r}-1\right)v_i^{(\ell)}.
  \end{equation*}
  This proves the eigenvalue formula.
  Taking real parts and using $\operatorname{Re}\omega_N^{\ell r}=\cos\left(2\pi\ell r/N\right)$ gives the stated expression for $\operatorname{Re}\lambda_\ell$.
\end{proof}

\begin{lemma}\label{lem:uniform_consensus_contraction}
  Let $b_{ij}:[0,\infty)\to[0,\infty)$ be continuous for $i,j\in[N]$ with $i\neq j$.
  Assume that there is a fixed strongly connected directed graph $G=\left([N],\mathcal E\right)$ and constants $b_\ast>0$ and $b^\ast>0$ such that
  \begin{equation*}
    b_\ast\leq b_{ij}(t)\leq b^\ast,\quad \left(j,i\right)\in\mathcal E,\quad t\geq0,
  \end{equation*}
  while $b_{ij}(t)=0$ whenever $\left(j,i\right)\notin\mathcal E$.
  If
  \begin{equation*}
    \dot z_i=\sum_{j\neq i}b_{ij}(t)\left(z_j-z_i\right),\quad i\in[N],
  \end{equation*}
  then there exist constants $C\geq1$ and $\lambda>0$, depending only on $G$, $b_\ast$, and $b^\ast$, such that
  \begin{equation*}
    \max_i z_i(t)-\min_i z_i(t)\leq C e^{-\lambda t}\left(\max_i z_i(0)-\min_i z_i(0)\right),\quad t\geq0.
  \end{equation*}
\end{lemma}

This lemma is a specialization of the exponential-consensus result in \cite[Cor. 4.6]{S-J}, with the decay estimate given in \cite[Eq. (4.21)]{S-J}.
The system above is the noiseless case of the model considered there, and the continuity assumption made here is stronger than the regularity hypothesis in \cite[Sec. 3.1]{S-J}.
For every edge $\left(j,i\right)\in\mathcal E$ and every $t\geq0$, one has $\int_t^{t+1}b_{ij}(s)\,ds\geq b_\ast$.
Thus every edge of $G$ is a $\delta$-arc on every interval $\left[t,t+1\right)$ with $\delta=b_\ast$.
Since $b_{ij}(t)=0$ whenever $\left(j,i\right)\notin\mathcal E$, the graph formed by the $\delta$-arcs is exactly the fixed strongly connected graph $G$ and is therefore quasi-strongly connected.
Hence the associated time-varying weighted graph is uniformly quasi-strongly $\delta$-connected with connectivity interval length $T=1$ in the sense of \cite[Def. 3.2]{S-J}.
Moreover, the bound $b_{ij}(t)\leq b^\ast$ implies
\begin{equation*}
  \int_{t_1}^{t_2}b_{ij}(s)\,ds\leq b^\ast\left(t_2-t_1\right),
\end{equation*}
so the integral boundedness hypothesis holds with $M_0=b^\ast$.
The cited exponential-consensus estimate therefore gives the stated constants $C$ and $\lambda$, which depend only on $G$, $b_\ast$, and $b^\ast$.

We conclude this appendix with the uniform trigonometric positivity estimate used in the proof of Lemma \ref{lem:general_q_parameter_estimates}.
We first verify the finitely many cases not covered by the uniform estimates using exact rational calculations.
For $0\leq\xi\leq2$, define
\begin{equation*}
  \begin{aligned}
    \underline{c}(\xi)&:=1-\frac{\xi^2}{2!}+\frac{\xi^4}{4!}-\frac{\xi^6}{6!},
    &\overline{c}(\xi)&:=\underline{c}(\xi)+\frac{\xi^8}{8!},\\
    \underline{s}(\xi)&:=\xi-\frac{\xi^3}{3!}+\frac{\xi^5}{5!}-\frac{\xi^7}{7!}+\frac{\xi^9}{9!}-\frac{\xi^{11}}{11!},
    &\overline{s}(\xi)&:=\underline{s}(\xi)+\frac{\xi^{11}}{11!}.
  \end{aligned}
\end{equation*}
The alternating-series remainder bounds give
\begin{equation}\label{eq:finite_taylor_bounds}
  \underline{c}(\xi)\leq\cos\xi\leq\overline{c}(\xi),\quad
  \underline{s}(\xi)\leq\sin\xi\leq\overline{s}(\xi),\quad 0\leq\xi\leq2.
\end{equation}

\begin{lemma}\label{lem:finite_exact_verification}
  Let $\mathcal M_0:=\left\{3,5,6,7,8,9,10,11,12\right\}$ and, for each $m\in\mathcal M_0$, set $n_m:=2m+1$.
  Then, for every $m\in\mathcal M_0$,
  \begin{equation*}
    \sum_{r=1}^{m}r^2\underline{c}\left(\frac{4r}{n_m}\right)>\frac{n_m^3}{475},\quad
    \sum_{r=1}^{m}r^4\overline{c}\left(\frac{4r}{n_m}\right)<-\frac{n_m^5}{5000}.
  \end{equation*}
  Let $\mathcal N_0:=\left\{7,11,13,15,17,19\right\}$ and, for $n\in\mathcal N_0$, define
  \begin{equation*}
    \nu_n:=\left\lfloor\frac{4n}{21}+\frac12\right\rfloor,\quad
    r_n:=\left\lfloor\frac{\pi n}{8}\right\rfloor,\quad
    x_n:=\frac{21}{8n}.
  \end{equation*}
  Then, for every $n\in\mathcal N_0$,
  \begin{equation*}
    \left(2\nu_n+1\right)\left(x_n-\frac{x_n^3}{6}\right)>1,\quad
    2\underline{s}(2)-\overline{s}\left(\frac{4\nu_n+2}{n}\right)-\overline{s}\left(\frac{4r_n+2}{n}\right)>0.
  \end{equation*}
\end{lemma}

\begin{proof}
  For each $m\in\mathcal M_0$, both quantities in the first display are rational.
  Exact substitution and comparison after clearing positive denominators give the first two inequalities, with the smallest margins occurring at $m=3$.
  The rational bounds $333/106<\pi<355/113$ determine
  \begin{equation*}
    \begin{aligned}
      \left(\nu_7,r_7\right)&=\left(1,2\right),\quad
      \left(\nu_{11},r_{11}\right)=\left(2,4\right),\quad
      \left(\nu_{13},r_{13}\right)=\left(2,5\right),\\
      \left(\nu_{15},r_{15}\right)&=\left(3,5\right),\quad
      \left(\nu_{17},r_{17}\right)=\left(3,6\right),\quad
      \left(\nu_{19},r_{19}\right)=\left(4,7\right).
    \end{aligned}
  \end{equation*}
  Once these floor values are fixed, both quantities in the second display are rational.
  Exact substitution and comparison after clearing positive denominators give the final two inequalities, with the smallest margins occurring at $n=13$ and $n=19$, respectively.
\end{proof}

\begin{lemma}\label{lem:uniform_trigonometric_positivity}
  Let $m\geq2$, set $n:=2m+1$, and suppose that $0<\varpi\leq4/n$.
  Define
  \begin{equation*}
    \Psi_{m,\varpi}(x):=\sum_{r=1}^{m}\cos\left(r\varpi\right)\left(1-\cos(rx)\right).
  \end{equation*}
  Then $\Psi_{m,\varpi}(x)>0$ for every $0<x<2\pi$.
\end{lemma}

\begin{proof}
  \textbf{Step 1. Reduction to the endpoint and the cases $\bm{m=2,4}$.}
  For fixed $x\in(0,2\pi)$,
  \begin{equation*}
    \frac{\partial}{\partial\varpi}\Psi_{m,\varpi}(x)
    =-\sum_{r=1}^{m}r\sin\left(r\varpi\right)\left(1-\cos(rx)\right)<0,
  \end{equation*}
  because $0<r\varpi\leq4m/\left(2m+1\right)<2<\pi$ and the $r=1$ term in the displayed derivative is strictly negative.
  It is therefore enough to prove $\Psi_{m,4/n}(x)>0$.
  Set $\kappa_r:=\cos\left(4r/n\right)$ for $r\in[m]$.
  Since $\Psi_{m,4/n}(2\pi-x)=\Psi_{m,4/n}(x)$, it suffices to consider $0<x\leq\pi$.
  Suppose first that $m=2$, so $n=5$, and set $y:=\cos x$.
  Since $\pi/2<8/5<\pi$, one has $\kappa_2<0$, while $0\leq1+y\leq2$.
  Using \eqref{eq:finite_taylor_bounds} and an exact rational calculation, we obtain
  \begin{equation*}
    \frac{\Psi_{2,4/5}(x)}{1-\cos x}
    =\kappa_1+2\kappa_2\left(1+y\right)
    \geq\underline{c}\left(\frac45\right)+4\underline{c}\left(\frac85\right)>0.
  \end{equation*}
  Hence $\Psi_{2,4/5}(x)>0$.
  Suppose next that $m=4$, so $n=9$, and set $s:=\left(1+\cos x\right)/2$.
  Since $s\in[0,1]$, we express the resulting cubic polynomial in the degree-three Bernstein basis, whose elements are nonnegative on this interval.
  Direct expansion gives
  \begin{equation*}
    \frac{\Psi_{4,4/9}(x)}{1-\cos x}
    =b_0\left(1-s\right)^3+3b_1s\left(1-s\right)^2+3b_2s^2\left(1-s\right)+b_3s^3,
  \end{equation*}
  where
  \begin{equation*}
    \begin{aligned}
      b_0&:=\kappa_1+\kappa_3,
      &b_1&:=\frac{3\kappa_1+4\kappa_2-5\kappa_3+16\kappa_4}{3},\\
      b_2&:=\frac{3\kappa_1+8\kappa_2+3\kappa_3-32\kappa_4}{3},
      &b_3&:=\kappa_1+4\kappa_2+9\kappa_3+16\kappa_4.
    \end{aligned}
  \end{equation*}
  Here $b_0>0$ and $b_2>0$ because $\kappa_1,\kappa_2,\kappa_3>0$ and $\kappa_4<0$.
  Applying \eqref{eq:finite_taylor_bounds} term by term gives
  \begin{equation*}
    \begin{aligned}
      b_1&\geq\frac{1}{3}\left(
      3\underline{c}\left(4/9\right)+4\underline{c}\left(8/9\right)
      -5\overline{c}\left(4/3\right)+16\underline{c}\left(16/9\right)\right)>0,\\
      b_3&\geq\underline{c}\left(4/9\right)+4\underline{c}\left(8/9\right)
      +9\underline{c}\left(4/3\right)+16\underline{c}\left(16/9\right)>0.
    \end{aligned}
  \end{equation*}
  Thus all four Bernstein coefficients are positive, and $\Psi_{4,4/9}(x)>0$.

  \medskip
  \noindent
  \textbf{Step 2. The small-$\bm{x}$ range.}
  We now assume that $m=3$ or $m\geq5$ and first consider
  \begin{equation}\label{eq:uniform_small_x_range}
    0<x\leq\frac{21}{4n}.
  \end{equation}
  Then $0<rx<21/8<3$ for $r\in[m]$, and the alternating-series remainder estimate gives
  \begin{equation*}
    \left|1-\cos(rx)-\frac{\left(rx\right)^2}{2!}+\frac{\left(rx\right)^4}{4!}\right|
    \leq\frac{\left(rx\right)^6}{6!}.
  \end{equation*}
  We first establish the following three estimates:
  \begin{equation}\label{eq:uniform_moment_estimates}
    \sum_{r=1}^{m}r^2\kappa_r>\frac{n^3}{475},\quad
    \sum_{r=1}^{m}r^4\kappa_r<-\frac{n^5}{5000},\quad
    \sum_{r=1}^{m}r^6<\frac{n^7}{896}.
  \end{equation}
  The third estimate follows from
  \begin{equation*}
    \sum_{r=1}^{m}r^6=\frac{n\left(n-1\right)\left(n+1\right)\left(3n^4-18n^2+31\right)}{2688}, \quad
    \frac{n^7}{896}-\sum_{r=1}^{m}r^6=\frac{n\left(21n^4-49n^2+31\right)}{2688}>0.
  \end{equation*}
  To prove the first two estimates, set $p_2(\xi):=\xi^2\cos(2\xi)$ and $p_4(\xi):=\xi^4\cos(2\xi)$.
  By the definitions of $p_2$, $p_4$, and $\kappa_r$,
  \begin{equation*}
    \sum_{r=1}^{m}r^2\kappa_r=\frac{n^3}{8}\frac{2}{n}\sum_{r=1}^{m}p_2\left(\frac{2r}{n}\right),\quad
    \sum_{r=1}^{m}r^4\kappa_r=\frac{n^5}{32}\frac{2}{n}\sum_{r=1}^{m}p_4\left(\frac{2r}{n}\right).
  \end{equation*}
  Here $2/n$ is the mesh width, and the points $2r/n$, $r\in[m]$, are the midpoints of the $m$ equal subintervals of $[1/n,1]$.
  Direct integration gives
  \begin{equation*}
    \int_0^1p_2(\xi)\,d\xi =-\int_0^1p_4(\xi)\,d\xi =\frac{\cos2}{2}+\frac{\sin2}{4} =:\mathcal I_0.
  \end{equation*}
  The alternating-series lower bounds for $\cos2$ through the $2^{10}/10!$ term and for $\sin2$ through the $2^{11}/11!$ term give
  \begin{equation*}
    \mathcal I_0 \geq\frac{1}{2}\left(1-\frac{2^2}{2!}+\frac{2^4}{4!}-\frac{2^6}{6!}+\frac{2^8}{8!}-\frac{2^{10}}{10!}\right)+\frac{1}{4}\underline{s}(2)
    =\frac{3001}{155925}>\frac{12}{625}.
  \end{equation*}
  Direct differentiation, together with $\left|1-2\xi^2\right|\leq1$ and $\xi^2\left(3-\xi^2\right)\leq2$ on $[0,1]$, gives
  \begin{equation*}
    \left|p_2''(\xi)\right|\leq10,\quad
    \left|p_4''(\xi)\right|\leq24,\quad 0\leq\xi\leq1.
  \end{equation*}
  Since $n\geq7$ under the present assumptions, one has
  $\cos(2\xi)\geq0$ on $[0,1/n]$.
  Hence
  \begin{equation*}
    0\leq\int_0^{1/n}p_2(\xi)\,d\xi \leq\frac{1}{3n^3}, \quad
    p_4(\xi)\geq0,\quad 0\leq\xi\leq\frac{1}{n}.
  \end{equation*}
  The composite midpoint estimate therefore gives
  \begin{equation*}
    \frac{2}{n}\sum_{r=1}^{m}p_2\left(\frac{2r}{n}\right) \geq\mathcal I_0-\frac{1}{3n^3}-\frac{5}{3n^2}, \quad
    \frac{2}{n}\sum_{r=1}^{m}p_4\left(\frac{2r}{n}\right) \leq-\mathcal I_0+\frac{4}{n^2}.
  \end{equation*}
  If $n\geq27$, then
  \begin{equation*}
    \frac{12}{625}-\frac{1}{3n^3}-\frac{5}{3n^2}>\frac{8}{475}, \quad
    -\frac{12}{625}+\frac{4}{n^2}<-\frac{4}{625}.
  \end{equation*}
  It follows that the first two estimates in \eqref{eq:uniform_moment_estimates} hold for $m\geq13$.
  If $m\in\mathcal M_0$, then $n=n_m$ and
  \begin{equation*}
    0<\frac{4r}{n}\leq\frac{4m}{2m+1}<2,\quad r\in[m].
  \end{equation*}
  Hence \eqref{eq:finite_taylor_bounds} gives
  \begin{equation*}
    \underline{c}\left(\frac{4r}{n}\right) \leq\kappa_r \leq\overline{c}\left(\frac{4r}{n}\right),\quad r\in[m].
  \end{equation*}
  Multiplying the lower bound by $r^2$, the upper bound by $r^4$, and summing over $r\in[m]$, the first two estimates follow from Lemma \ref{lem:finite_exact_verification}.
  Thus all three estimates in \eqref{eq:uniform_moment_estimates} hold for every $m=3$ or $m\geq5$.
  Since $\left|\kappa_r\right|\leq1$, the Taylor remainder estimate and \eqref{eq:uniform_moment_estimates} yield
  \begin{equation*}
    \Psi_{m,4/n}(x)>\frac{n^3x^2}{950}+\frac{n^5x^4}{120000}-\frac{n^7x^6}{645120}.
  \end{equation*}
  The right-hand side can be written as
  \begin{equation*}
    n\left(nx\right)^2 \left[\frac{1}{950}+\frac{\left(nx\right)^2}{120000}-\frac{\left(nx\right)^4}{645120}\right].
  \end{equation*}
  As a function of $z:=\left(nx\right)^2$, the expression in brackets is a concave quadratic polynomial.
  It therefore attains its minimum on $0\leq z\leq\left(21/4\right)^2$ at an endpoint.
  Its value at $z=0$ is $1/950$, while an exact rational calculation shows that its value at $z=\left(21/4\right)^2$ is greater than $1/10000$.
  Hence $\Psi_{m,4/n}(x)>0$ throughout \eqref{eq:uniform_small_x_range}.

  \medskip
  \noindent
  \textbf{Step 3. The remaining $\bm{x}$-range.}
  It remains to consider
  \begin{equation}\label{eq:uniform_remaining_x_range}
    \frac{21}{4n}\leq x\leq\pi.
  \end{equation}
  Set $r_0:=\left\lfloor\pi n/8\right\rfloor$, $a_r:=1-\cos(rx)$, and $S_k:=\sum_{r=1}^{k}a_r$.
  The sequence $\kappa_r$ is strictly decreasing, with $\kappa_r\geq0$ for $1\leq r\leq r_0$ and $\kappa_r<0$ for $r_0<r\leq m$.
  The finite cosine-sum identity gives
  \begin{equation}\label{eq:uniform_partial_sum_lower_bound}
    S_k=k+\frac12-\frac{\sin\left(\left(2k+1\right)x/2\right)}{2\sin(x/2)}
    \geq k+\frac12-\frac{1}{2\sin(x/2)}.
  \end{equation}
  Suppose first that $n\geq21$ and set $\nu_0:=\left\lfloor4n/21+1/2\right\rfloor$ and $x_0:=21/\left(8n\right)$.
  Since $n$ is odd, the fractional part of $4n/21+1/2$ is a nonzero multiple of $1/42$, and hence $\nu_0\geq4n/21-1/2+1/42$.
  Using $\sin x_0\geq x_0-x_0^3/6$, we obtain
  \begin{equation*}
    \frac{1}{2\sin x_0}-\frac{1}{2x_0}\leq\frac{x_0}{12\left(1-x_0^2/6\right)}\leq\frac4{383}<\frac1{42}.
  \end{equation*}
  Since $1/\left(2x_0\right)=4n/21$, it follows that $1/\left(2\sin x_0\right)-1/2\leq\nu_0$.
  Also, $\nu_0\leq4n/21+1/2\leq3n/14<\pi n/8$, so $\nu_0\leq r_0$.
  By \eqref{eq:uniform_remaining_x_range} and the monotonicity of sine on $[0,\pi/2]$,
  \begin{equation*}
    \frac{1}{2\sin(x/2)}-\frac12\leq\nu_0,
  \end{equation*}
  and \eqref{eq:uniform_partial_sum_lower_bound} gives $S_k\geq\max\left\{0,k-\nu_0\right\}$.
  Since $\kappa_k-\kappa_{k+1}>0$, summation by parts and the lower bound for $S_k$ give
  \begin{equation*}
    \sum_{r=1}^{r_0}\kappa_ra_r
    =\kappa_{r_0}S_{r_0}+\sum_{k=1}^{r_0-1}\left(\kappa_k-\kappa_{k+1}\right)S_k
    \geq\sum_{r=\nu_0+1}^{r_0}\kappa_r.
  \end{equation*}
  Since $0\leq a_r\leq2$ and $\kappa_r<0$ for $r>r_0$,
  \begin{equation*}
    \Psi_{m,4/n}(x)\geq\sum_{r=\nu_0+1}^{r_0}\kappa_r+2\sum_{r=r_0+1}^{m}\kappa_r.
  \end{equation*}
  For $\mathcal T(a):=\sum_{r=a}^{m}\cos\left(4r/n\right)$, the finite cosine-sum formula gives
  \begin{equation*}
    \mathcal T(a)=\frac{\sin2-\sin\left(\left(4a-2\right)/n\right)}{2\sin\left(2/n\right)}.
  \end{equation*}
  Consequently,
  \begin{equation*}
    \Psi_{m,4/n}(x)\geq\frac{2\sin2-\sin\left(\left(4\nu_0+2\right)/n\right)-\sin\left(\left(4r_0+2\right)/n\right)}{2\sin\left(2/n\right)}.
  \end{equation*}
  Here $\left(4\nu_0+2\right)/n\leq20/21$ and $\sin\left(\left(4r_0+2\right)/n\right)\leq1$.
  The alternating-series bounds give $2\sin2-1-\sin\left(20/21\right)>0$.
  Thus $\Psi_{m,4/n}(x)>0$ when $n\geq21$.
  It remains to treat $m\in\left\{3,5,6,7,8,9\right\}$ in \eqref{eq:uniform_remaining_x_range}.
  Then $n\in\mathcal N_0$, and Lemma \ref{lem:finite_exact_verification} gives
  \begin{equation*}
    \left(2\nu_n+1\right)\left(x_n-\frac{x_n^3}{6}\right)>1.
  \end{equation*}
  Since $x/2\geq x_n$ and the sine function is increasing on $\left[0,\pi/2\right]$, the alternating Taylor bound gives
  \begin{equation*}
    \sin\left(\frac{x}{2}\right)
    \geq\sin x_n
    \geq x_n-\frac{x_n^3}{6}
    >\frac{1}{2\nu_n+1}.
  \end{equation*}
  Hence
  \begin{equation*}
    \frac{1}{2\sin(x/2)}-\frac12<\nu_n.
  \end{equation*}
  The preceding summation-by-parts argument applies with $\nu_0=\nu_n$ and $r_0=r_n$.
  The values of $\left(\nu_n,r_n\right)$ listed in the proof of Lemma \ref{lem:finite_exact_verification} show that
  \begin{equation*}
    0\leq\frac{4\nu_n+2}{n}\leq2,\quad
    0\leq\frac{4r_n+2}{n}\leq2.
  \end{equation*}
  Lemma \ref{lem:finite_exact_verification} and \eqref{eq:finite_taylor_bounds} therefore give
  \begin{equation*}
    2\sin2-\sin\left(\frac{4\nu_n+2}{n}\right)-\sin\left(\frac{4r_n+2}{n}\right)
    \geq 2\underline{s}(2)-\overline{s}\left(\frac{4\nu_n+2}{n}\right)-\overline{s}\left(\frac{4r_n+2}{n}\right)>0.
  \end{equation*}
  Hence the right-hand side of the preceding lower bound is strictly positive, and therefore $\Psi_{m,4/n}(x)>0$.
  This completes the proof.
\end{proof}

\section{Deferred proofs of model-specific lemmas}\label{sec:model_lemma_proofs}

\setcounter{equation}{0}

This appendix contains the deferred proofs of the model-specific lemmas stated in Section \ref{sec:stability}.
It also records the auxiliary cosine-sum representation used in the proofs for the $2$-twisted classification.
The deferred proofs are arranged in the order in which the corresponding lemmas appear in Section \ref{sec:stability}.
We first treat the Fourier and quotient framework and the general parameter estimates, and then the lemmas for the $1$-twisted and $2$-twisted classifications.

\begin{proof}[\textbf{Proof of Lemma \ref{lem:forward_model_fourier_reduction}}]
  This is Lemma \ref{lem:circulant_diagonalization} with $c_r=K\cos\left(2\pi q r/N\right)$ for $r\in[m]$.
  The expression for the real part follows from the same lemma.
  The phase-shift mode is $v^{(0)}=\mathbf 1$.
  If $\ell\in[N-1]$, Lemma \ref{lem:finite_fourier_basis} gives
  \begin{equation*}
    \sum_{i=1}^{N}v_i^{(\ell)}=0.
  \end{equation*}
  Taking real and imaginary parts shows that the corresponding real Fourier subspace has zero sum and is therefore transverse to $\operatorname{span}\left\{\mathbf 1\right\}$.
  Invariance follows from the eigenvalue relation and the conjugate symmetry of the real operator.
\end{proof}

\begin{proof}[\textbf{Proof of Lemma \ref{lem:quotient_linearized_criterion}}]
  The phase-shift invariance of the vector field induces a smooth quotient system on $\mathcal Q$, in which the orbit $\mathcal O_\phi$ is represented by the equilibrium $\Pi(0)$.
  The linearization of this quotient system is the operator induced by $J$ on $\mathcal Q$.
  To represent the quotient locally, choose a complementary hyperplane $H$ to $\operatorname{span}\left\{\mathbf 1\right\}$ and write $u=\sigma\mathbf 1+h$, where $h\in H$.
  For example, one may take
  \begin{equation*}
    H:=\left\{h\in\mathbb R^N:\sum_{i=1}^{N}h_i=0\right\}.
  \end{equation*}
  The restriction of $\Pi$ to $H$ is a linear isomorphism from $H$ onto $\mathcal Q$. The hyperplane $H$ is used only as a coordinate slice and need not be invariant under the directed dynamics.
  If $0$ is a simple eigenvalue of $J$ and every other eigenvalue has strictly negative real part, then the quotient spectrum lies in the open left half-plane.
  Lemma \ref{lem:standard_linearized_criterion} gives exponential stability of the reduced equilibrium.
  Since $\left\|\cdot\right\|_D$ is a norm on the finite-dimensional space $\mathcal Q$, equivalent to every other norm, this is precisely local orbital exponential stability in quotient diameter.
  If $J$ has an eigenvalue with strictly positive real part, the corresponding spectral direction is not contained in $\operatorname{span}\left\{\mathbf 1\right\}$ and therefore descends to an unstable direction in $\mathcal Q$.
  Lemma \ref{lem:standard_linearized_criterion} then gives nonlinear instability of the reduced equilibrium, and hence nonlinear instability of $\mathcal O_\phi$ in the quotient sense.
\end{proof}

\begin{proof}[\textbf{Proof of Lemma \ref{lem:general_q_parameter_estimates}}]
  Set
  \begin{equation*}
    g_q:=\gcd(N,q),\quad N_q:=\frac{N}{g_q},\quad
    \widehat q:=\frac{q}{g_q},\quad
    \widehat q^\sharp:=\min\left\{\widehat q,N_q-\widehat q\right\}.
  \end{equation*}
  Then $\gcd\left(N_q,\widehat q\right)=1$ and $\widehat q^\sharp=q^\sharp/g_q$.
  We first prove part (1).
  Set $\varpi_q:=2\pi q^\sharp/N$.
  Since $q\equiv\pm q^\sharp\pmod N$ and the cosine function is even,
  \begin{equation*}
    \Gamma_{N,m,q}^{(\ell)}
    =\sum_{r=1}^{m}\cos\left(r\varpi_q\right)\left(1-\cos\left(\frac{2\pi\ell r}{N}\right)\right),
    \quad \ell\in[N-1].
  \end{equation*}
  If $m=1$, the assumed lower bound on $N$ gives $0<\varpi_q\leq4/3<\pi/2$.
  Hence $\cos\left(\varpi_q\right)>0$, and the displayed sum is strictly positive for every $\ell\in[N-1]$.
  If $m\geq2$, the parameter bound gives $0<\varpi_q\leq4/\left(2m+1\right)$.
  The displayed sum is then $\Psi_{m,\varpi_q}\left(2\pi\ell/N\right)$, so Lemma \ref{lem:uniform_trigonometric_positivity} yields $\Gamma_{N,m,q}^{(\ell)}>0$ for every $\ell\in[N-1]$.
  To verify existence, divide the parameter bound by $g_q$ to obtain
  \begin{equation*}
    N_q\geq\frac{\pi}{2}\left(2m+1\right)\widehat q^\sharp
    \geq\frac{\pi}{2}\left(2m+1\right)>m+1.
  \end{equation*}
  Hence $N_q$ divides neither $m$ nor $m+1$.
  Since $\gcd\left(N_q,\widehat q\right)=1$, this is equivalent to $N\nmid mq$ and $N\nmid\left(m+1\right)q$.
  Thus the existence conditions also hold.

  We next prove part (2).
  Dividing the assumed inequality by $g_q$ gives
  \begin{equation*}
    N_q\leq\frac{8m+3}{3}.
  \end{equation*}
  The two divisibility assumptions are equivalent to $N_q\nmid m$ and $N_q\nmid m+1$ and in particular imply $N_q\geq3$.
  Define the $N_q$-periodic sequence
  \begin{equation*}
    h_r:=\cos\left(\frac{2\pi\widehat q r}{N_q}\right)
    \left(1-\cos\left(\frac{2\pi\widehat q r}{N_q}\right)\right),\quad r\in\mathbb Z.
  \end{equation*}
  Then $\Gamma_{N,m,q}^{(q)}=\sum_{r=1}^{m}h_r$ and $h_{N_q-r}=h_r$.
  Since multiplication by $\widehat q$ permutes the nonzero residue classes modulo $N_q$,
  \begin{equation*}
    \sum_{r=1}^{N_q-1}h_r
    =\sum_{r=1}^{N_q-1}\cos\left(\frac{2\pi r}{N_q}\right)
    -\sum_{r=1}^{N_q-1}\cos^2\left(\frac{2\pi r}{N_q}\right)
    =-\frac{N_q}{2}.
  \end{equation*}
  Moreover,
  \begin{equation}\label{eq:general_q_primitive_pointwise_bounds}
    -2\leq h_r\leq\frac14,
  \end{equation}
  because $u\left(1-u\right)\in[-2,1/4]$ for $u\in[-1,1]$.
  Write $m=a_0N_q+s$, where $a_0\geq0$ and $0\leq s\leq N_q-1$.
  The divisibility exclusions give $1\leq s\leq N_q-2$.
  If $a_0\geq1$, then
  \begin{equation*}
    \sum_{r=1}^{m}h_r=-\frac{a_0N_q}{2}+\sum_{r=1}^{s}h_r
    \leq-\frac{N_q}{2}+\frac{N_q-2}{4}=-\frac{N_q+2}{4}\leq-\frac54.
  \end{equation*}
  It remains to consider $a_0=0$, so $m=s\leq N_q-2$.
  Suppose first that $N_q=2M+1$.
  Symmetry gives $\sum_{r=1}^{M}h_r=-N_q/4$.
  If $m\geq M$, then
  \begin{equation*}
    \sum_{r=1}^{m}h_r\leq-\frac{N_q}{4}+\frac{m-M}{4}
    \leq-\frac{M+2}{4}\leq-\frac34.
  \end{equation*}
  If $m<M$, then \eqref{eq:general_q_primitive_pointwise_bounds} gives
  \begin{equation*}
    \sum_{r=1}^{m}h_r=-\frac{N_q}{4}-\sum_{r=m+1}^{M}h_r
    \leq\frac{3N_q}{4}-1-2m\leq-\frac14,
  \end{equation*}
  where the last inequality follows from $3N_q\leq8m+3$.

  Suppose next that $N_q=2M$.
  Since $\widehat q$ is odd, $h_M=-2$, and symmetry gives $\sum_{r=1}^{M}h_r=-N_q/4-1$.
  If $m\geq M$, then
  \begin{equation*}
    \sum_{r=1}^{m}h_r\leq-\frac{N_q}{4}-1+\frac{m-M}{4}
    \leq-\frac{M+6}{4}\leq-2.
  \end{equation*}
  If $m<M$, then
  \begin{equation*}
    \sum_{r=1}^{m}h_r=-\frac{N_q}{4}-1-\sum_{r=m+1}^{M}h_r
    \leq\frac{3N_q}{4}-1-2m\leq-\frac14.
  \end{equation*}
  Thus $\Gamma_{N,m,q}^{(q)}\leq-1/4$ in every case.
\end{proof}

\begin{proof}[\textbf{Proof of Lemma \ref{lem:q1_first_mode_positive_range}}]
  First suppose that $N\geq4m+1$.
  Then $0<2\pi r/N<\pi/2$ for every $r\in[m]$.
  Hence $\cos\left(2\pi r/N\right)>0$ and $1-\cos\left(2\pi r/N\right)>0$ for every $r\in[m]$, and therefore $\Gamma_{N,m,1}^{(1)}>0$.
  It remains to consider $3m+2\leq N\leq4m$.
  Set $n:=2m+1$, $y:=\pi/N$, and $\upsilon:=ny$.
  Then $N\geq3m+2=\left(3n+1\right)/2$ and $N\leq4m=2n-2$.
  Hence $\pi/2<\upsilon<2\pi/3$.
  In particular, $\cos\upsilon<0$.
  Since $0<\cos y<1$, we have
  \begin{equation*}
    2-\frac{\cos\upsilon}{\cos y}>2-\cos\upsilon.
  \end{equation*}
  Also, $\sin y<y=\upsilon/n$.
  Using \eqref{eq:q1_first_mode_gamma_closed_form_ny}, we obtain
  \begin{equation*}
    4\Gamma_{N,m,1}^{(1)}
    >n\left[\frac{\sin\upsilon}{\upsilon}\left(2-\cos\upsilon\right)-1\right].
  \end{equation*}
  Define $H(x):=\sin x\left(2-\cos x\right)-x$.
  Then
  \begin{equation*}
    H'(x)=2\cos x\left(1-\cos x\right)\leq0,\quad x\in\left[\frac{\pi}{2},\frac{2\pi}{3}\right].
  \end{equation*}
  Therefore $H$ is decreasing on this interval, and
  \begin{equation*}
    H(x)\geq H\left(\frac{2\pi}{3}\right)=\frac{5\sqrt{3}}{4}-\frac{2\pi}{3}>0.
  \end{equation*}
  Taking $x=\upsilon$ gives
  \begin{equation*}
    \frac{\sin\upsilon}{\upsilon}\left(2-\cos\upsilon\right)>1.
  \end{equation*}
  Thus $\Gamma_{N,m,1}^{(1)}>0$.
\end{proof}

\begin{proof}[\textbf{Proof of Lemma \ref{lem:q1_first_mode_negative_range}}]
  We first handle the small values of $m$.
  There is no integer $N$ satisfying $2m+2\leq N\leq2.9m$ when $m=2$.
  If $m=3$, then the only such value is $N=8$.
  If $m=4$, then the only such values are $N=10$ and $N=11$.
  For a fixed $N$, set
  \begin{equation*}
    T_1:=\sum_{r=1}^{N-1}\cos\left(\frac{2\pi r}{N}\right)\left(1-\cos\left(\frac{2\pi r}{N}\right)\right).
  \end{equation*}
  By discrete orthogonality, $T_1=-1-(N-2)/2=-N/2$.
  When $N=2m+2$, the unique middle index is $r=N/2$, and its contribution to $T_1$ is $-2$.
  Pairing the first and last $m$ terms gives $T_1=2\Gamma_{N,m,1}^{(1)}-2$.
  Hence $\Gamma_{N,m,1}^{(1)}=1-N/4<0$.
  This covers $\left(m,N\right)=\left(3,8\right)$ and $\left(4,10\right)$.
  For $\left(m,N\right)=\left(4,11\right)$, the middle indices are $5$ and $6$.
  The same decomposition gives
  \begin{equation*}
    \Gamma_{11,4,1}^{(1)}
    =-\frac{11}{4}
    +\cos\left(\frac{\pi}{11}\right)\left(1+\cos\left(\frac{\pi}{11}\right)\right)<-\frac{11}{4}+2<0.
  \end{equation*}
  It remains to consider $m\geq5$.
  Set $n:=2m+1$, $y:=\pi/N$, and $\upsilon:=ny$.
  Since $N\geq2m+2=n+1$, we have $0<\upsilon<\pi$.
  Since $N\leq2.9m=29m/10$, we also have
  \begin{equation*}
    \upsilon=\frac{\left(2m+1\right)\pi}{N}
    \geq \frac{10\left(2m+1\right)\pi}{29m}
    =\left(\frac{20}{29}+\frac{10}{29m}\right)\pi>\frac{20\pi}{29}.
  \end{equation*}
  Thus $20\pi/29<\upsilon<\pi$.
  Moreover, $y\leq\pi/\left(2m+2\right)\leq\pi/12$.
  We claim that, for $x\in[2\pi/3,\pi]$ and $1<\beta<2$, the function
  \begin{equation*}
    f_\beta(x):=\frac{\sin x}{x}\left(2-\beta\cos x\right)
  \end{equation*}
  is strictly decreasing.
  Indeed, writing $c:=-\cos x$ and $s:=\sin x$, one has $c\geq1/2$ and
  \begin{equation*}
    x^2f_\beta'(x)=\beta xs^2-\left(xc+s\right)\left(2+\beta c\right).
  \end{equation*}
  Since $s^2=1-c^2$ and $xc+s\geq xc$, it is enough to prove $\beta\left(1-c^2\right)<c\left(2+\beta c\right)$.
  This follows from $\beta<2$ and $c\geq1/2$.
  Hence $f_\beta'(x)<0$.
  Using $\sin y\geq y\left(1-y^2/6\right)$ and $\cos y\geq1-y^2/2$ on $0<y\leq\pi/12$, set
  \begin{equation*}
    \alpha_0:=\frac{1}{1-\pi^2/864},\quad \beta_0:=\frac{1}{1-\pi^2/288}.
  \end{equation*}
  Then $1<\beta_0<2$, $y/\sin y\leq\alpha_0$, and $1/\cos y\leq\beta_0$.
  Since $\cos\upsilon<0$, we obtain
  \begin{equation*}
    \frac{1}{n}\frac{\sin\upsilon}{\sin y}
    \left(2-\frac{\cos\upsilon}{\cos y}\right)
    \leq\alpha_0\frac{\sin\upsilon}{\upsilon}\left(2-\beta_0\cos\upsilon\right)
    =\alpha_0 f_{\beta_0}(\upsilon).
  \end{equation*}
  The monotonicity just proved and $\upsilon>20\pi/29$ imply
  \begin{equation*}
    \alpha_0 f_{\beta_0}(\upsilon)<\alpha_0 f_{\beta_0}\left(\frac{20\pi}{29}\right).
  \end{equation*}
  Put $z:=9\pi/29$.
  Since $\sin\left(20\pi/29\right)=\sin z$ and $-\cos\left(20\pi/29\right)=\cos z$, we have
  \begin{equation*}
    \alpha_0 f_{\beta_0}\left(\frac{20\pi}{29}\right)
    =\alpha_0\frac{9}{20}\frac{\sin z}{z}\left(2+\beta_0\cos z\right).
  \end{equation*}
  Using $333/106<\pi<355/113$, set
  \begin{equation*}
    \bar{\alpha}_0:=\left(1-\frac{1}{864}\left(\frac{355}{113}\right)^2\right)^{-1},\quad
    \bar{\beta}_0:=\left(1-\frac{1}{288}\left(\frac{355}{113}\right)^2\right)^{-1}.
  \end{equation*}
  Then $\alpha_0<\bar{\alpha}_0<175/173$ and $\beta_0<\bar{\beta}_0<29/28$.
  Also, since $z=9\pi/29$ and $333/106<\pi<355/113$, we have $77/79<z<1$.
  Using the alternating-series remainder bounds for the Maclaurin expansions of $\sin z/z$ and $\cos z$, and noting that
  \begin{equation*}
    1-\frac{\xi^2}{6}+\frac{\xi^4}{120}
    \quad\text{and}\quad
    1-\frac{\xi^2}{2}+\frac{\xi^4}{24}
  \end{equation*}
  are decreasing on $\left[0,1\right]$, we obtain
  \begin{equation*}
    \frac{\sin z}{z}
    \leq 1-\frac{1}{6}\left(\frac{77}{79}\right)^2
    +\frac{1}{120}\left(\frac{77}{79}\right)^4
    <\frac{17}{20}
  \end{equation*}
  and
  \begin{equation*}
    \cos z
    \leq 1-\frac{1}{2}\left(\frac{77}{79}\right)^2
    +\frac{1}{24}\left(\frac{77}{79}\right)^4
    <\frac{563}{1000}.
  \end{equation*}
  Therefore
  \begin{equation*}
    \alpha_0 f_{\beta_0}\left(\frac{20\pi}{29}\right)
    <\frac{175}{173}\cdot\frac{9}{20}\cdot\frac{17}{20}
    \left(2+\frac{29}{28}\cdot\frac{563}{1000}\right)
    =\frac{11066031}{11072000}<1.
  \end{equation*}
  Consequently,
  \begin{equation*}
    \frac{1}{n}\frac{\sin\upsilon}{\sin y}
    \left(2-\frac{\cos\upsilon}{\cos y}\right)<1.
  \end{equation*}
  By \eqref{eq:q1_first_mode_gamma_closed_form_ny}, this gives $\Gamma_{N,m,1}^{(1)}<0$.
\end{proof}

\begin{proof}[\textbf{Proof of Lemma \ref{lem:q1_intermediate_sign_reduction}}]
  Fix $2\leq\ell\leq N-2$.
  Set $x:=2\pi/N$, and define
  \begin{equation*}
    T_\ell:=\sum_{r=1}^{N-1}\cos(rx)\left(1-\cos(\ell rx)\right).
  \end{equation*}
  We first establish the full-cycle cancellation.
  By Lemma \ref{lem:finite_fourier_basis}, for every integer $k\not\equiv0\pmod N$,
  \begin{equation*}
    \sum_{r=1}^{N-1}\cos(k rx)=-1.
  \end{equation*}
  Also,
  \begin{equation*}
    \cos(rx)\cos(\ell rx)=\frac{1}{2}\left[\cos\left(\left(\ell-1\right)rx\right)+\cos\left(\left(\ell+1\right)rx\right)\right].
  \end{equation*}
  Since $2\leq\ell\leq N-2$, both $\ell-1$ and $\ell+1$ are nonzero modulo $N$.
  Hence
  \begin{equation*}
    \sum_{r=1}^{N-1}\cos(rx)\cos(\ell rx)=-1.
  \end{equation*}
  Therefore $T_\ell=0$.
  We next derive the complementary-sum representation.
  Decompose $\{1,\dots,N-1\}$ into the three ranges $\{1,\dots,m\}$, $\{m+1,\dots,N-m-1\}$, and $\{N-m,\dots,N-1\}$.
  Pairing the first and third ranges through $r\leftrightarrow N-r$ gives
  \begin{equation*}
    T_\ell=2\Gamma_{N,m,1}^{(\ell)}+\sum_{r=m+1}^{N-m-1}\cos(rx)\left(1-\cos(\ell rx)\right).
  \end{equation*}
  Since $T_\ell=0$, we obtain \eqref{eq:q1_intermediate_complement_formula}.
  The same representation yields nonnegativity in the stated parameter range. Every index $r\in\{m+1,\dots,N-m-1\}$ satisfies $N/4<r<3N/4$. Indeed, $r\geq m+1>N/4$ and $r\leq N-m-1<3N/4$, because $N\leq4m$. Consequently, $-\cos(rx)>0$ throughout the middle range, while $1-\cos(\ell rx)\geq0$. The complementary-sum formula therefore implies $\Gamma_{N,m,1}^{(\ell)}\geq0$.

  We now prove the strictness assertion. Suppose first that the middle range contains at least two indices. Since it is an interval of integers, it contains two consecutive indices $r$ and $r+1$. If $\Gamma_{N,m,1}^{(\ell)}=0$, then every nonnegative summand in \eqref{eq:q1_intermediate_complement_formula} must vanish. In particular,
  \begin{equation*}
    \cos(\ell rx)=\cos\left(\ell\left(r+1\right)x\right)=1.
  \end{equation*}
  Thus $N$ divides both $\ell r$ and $\ell\left(r+1\right)$, and hence $N$ divides $\ell$, contradicting $2\leq\ell\leq N-2$. Therefore $\Gamma_{N,m,1}^{(\ell)}>0$ whenever the middle range contains at least two indices.

  It remains to consider the case in which the middle range contains exactly one index. Then $N=2m+2$, and the unique middle index is $r=N/2$. Using Lemma~\ref{lem:finite_fourier_basis} with the first and second Fourier frequencies gives
  \begin{equation*}
    \sum_{r=1}^{N-1}\cos(rx)\left[1-\cos(rx)\right]=-1-\frac{1}{2}\left[\left(N-1\right)-1\right]=-\frac{N}{2}.
  \end{equation*}
  The unique middle term is $\cos(\pi)\left[1-\cos(\pi)\right]=-2$. Hence the same complementary decomposition yields
  \begin{equation*}
    2\Gamma_{N,m,1}^{(1)}-2=-\frac{N}{2},\quad \Gamma_{N,m,1}^{(1)}=1-\frac{N}{4}\leq0.
  \end{equation*}
  Thus the one-index case is incompatible with the additional assumption $\Gamma_{N,m,1}^{(1)}>0$.

  Therefore, whenever $\Gamma_{N,m,1}^{(1)}>0$, the middle range contains at least two indices, and the preceding argument gives $\Gamma_{N,m,1}^{(\ell)}>0$ for every $2\leq\ell\leq N-2$.
\end{proof}

\begin{proof}[\textbf{Proof of Lemma \ref{lem:q1_first_mode_zero_41}}]
  For sufficiency, if $\left(N,m\right)=\left(4,1\right)$, then
  \begin{equation*}
    \Gamma_{4,1,1}^{(1)}=\cos\left(\frac{\pi}{2}\right)\left(1-\cos\left(\frac{\pi}{2}\right)\right)=0.
  \end{equation*}
  For necessity, suppose that $\Gamma_{N,m,1}^{(1)}=0$ and set $x:=2\pi/N$. If $m=1$, then $\cos\left(x\right)\left(1-\cos\left(x\right)\right)=0$. 
  Since $0<x\leq2\pi/3$, the factor $1-\cos\left(x\right)$ is nonzero, and hence $\cos\left(x\right)=0$. 
  The only zero of the cosine function in this interval is $x=\pi/2$, so $N=4$.
  Suppose now that $m\geq2$. Set $n:=2m+1$ and define
  \begin{equation*}
    \Phi_m\left(\vartheta\right):=\sum_{r=1}^{m}\cos\left(r\vartheta\right)\left(1-\cos\left(r\vartheta\right)\right).
  \end{equation*}
  Then $\Phi_m\left(x\right)=0$ and $n\geq5$. Applying the finite cosine-sum formula at $\vartheta$ and $2\vartheta$, for $0<\vartheta<\pi$ we obtain
  \begin{equation}\label{eq:q1_zero_41_identity}
    4\Phi_m\left(\vartheta\right)=2\frac{\sin\left(n\vartheta/2\right)}{\sin\left(\vartheta/2\right)}-\frac{\sin\left(n\vartheta\right)}{\sin\left(\vartheta\right)}-n
    =\frac{\sin\left(n\vartheta/2\right)}{\sin\left(\vartheta/2\right)}\left(2-\frac{\cos\left(n\vartheta/2\right)}{\cos\left(\vartheta/2\right)}\right)-n.
  \end{equation}
  If $N\leq n$, then $\pi\leq nx/2<2\pi$, where the upper bound follows from $n=2m+1\leq2N-3<2N$. Since $m\geq2$ and $m\leq N-2$, we have $N\geq4$ and $0<x/2\leq\pi/4$. 
  Therefore $\sin\left(x/2\right)>0$, $2\cos\left(x/2\right)>1$, $\sin\left(nx/2\right)\leq0$, and $\cos\left(nx/2\right)\leq1$. 
  It follows that $2-\cos\left(nx/2\right)/\cos\left(x/2\right)>0$, so \eqref{eq:q1_zero_41_identity} gives $\Phi_m\left(x\right)<0$, contradicting $\Phi_m\left(x\right)=0$.
  Hence $N\geq2m+2$. On the other hand, Lemma \ref{lem:q1_first_mode_positive_range} excludes $N\geq3m+2$. 
  Therefore
  \begin{equation}\label{eq:q1_zero_41_range}
    2m+2\leq N\leq3m+1.
  \end{equation}
  We claim that $N=6$. Once this is proved, \eqref{eq:q1_zero_41_range} and $m\geq2$ give $m=2$, whereas
  \begin{equation*}
    \Gamma_{6,2,1}^{(1)}=\frac{1}{2}\left(1-\frac{1}{2}\right)-\frac{1}{2}\left(1+\frac{1}{2}\right)=-\frac{1}{2},
  \end{equation*}
  contradicting the assumption $\Gamma_{N,m,1}^{(1)}=0$ and thereby establishing necessity. 

  It remains to prove the claim. Define
  \begin{equation*}
    \mathcal P_0\left(X\right):=2,\quad \mathcal P_1\left(X\right):=X,\quad 
    \mathcal P_{r+1}\left(X\right):=X\mathcal P_r\left(X\right)-\mathcal P_{r-1}\left(X\right),\quad r\geq1.
  \end{equation*}
  A direct induction shows that $\mathcal P_r\in\mathbb Z[X]$ and that $\mathcal P_r$ is monic of degree $r$ for $r\geq1$. 
  A second induction gives the associated trigonometric representation. 
  The cases $r=0,1$ are immediate, and the induction step follows from the defining relation together with 
  $2\cos\left(\vartheta\right)\cos\left(r\vartheta\right)=\cos\left(\left(r+1\right)\vartheta\right)+\cos\left(\left(r-1\right)\vartheta\right)$. 
  Therefore $\mathcal P_r\left(2\cos\left(\vartheta\right)\right)=2\cos\left(r\vartheta\right)$ for every $r\geq0$. Set
  \begin{equation*}
    \mathcal H_m\left(X\right):=2\sum_{r=1}^{m}\mathcal P_r\left(X\right)-\sum_{r=1}^{m}\mathcal P_{2r}\left(X\right)-2m,\quad \mathcal K_N\left(X\right):=\mathcal P_N\left(X\right)-2.
  \end{equation*}
  Then $\mathcal H_m,\mathcal K_N\in\mathbb Z[X]$, the polynomial $\mathcal K_N$ is monic of degree $N$, and
  \begin{equation}\label{eq:q1_zero_41_polynomial_values}
    \mathcal H_m\left(2\cos\left(\vartheta\right)\right)=4\Phi_m\left(\vartheta\right),\quad 
    \mathcal K_N\left(2\cos\left(\vartheta\right)\right)=2\cos\left(N\vartheta\right)-2.
  \end{equation}
  To determine the zeros of $\mathcal K_N$, suppose that $\mathcal K_N\left(X_0\right)=0$ and choose a complex root $z$ of $z^2-X_0z+1=0$. 
  Then $z\neq0$ and $X_0=z+z^{-1}$. Since $\mathcal P_0\left(X_0\right)=2$ and $\mathcal P_1\left(X_0\right)=z+z^{-1}$, a direct induction using the defining relation gives 
  $\mathcal P_r\left(X_0\right)=z^r+z^{-r}$ for every $r\geq0$. 
  Hence $0=\mathcal K_N\left(X_0\right)=z^N+z^{-N}-2=z^{-N}\left(z^N-1\right)^2$, so $z^N=1$. 
  Thus $z=\exp\left(2\pi\mathrm i j/N\right)$ and $X_0=2\cos\left(2\pi j/N\right)$ for some $j\in\left\{0,1,\ldots,N-1\right\}$. 
  Every value of this form is a zero by \eqref{eq:q1_zero_41_polynomial_values}, and therefore
  \begin{equation}\label{eq:q1_zero_41_K_roots}
    \mathcal K_N\left(X_0\right)=0\quad\Longleftrightarrow\quad X_0=2\cos\left(\frac{2\pi j}{N}\right)\quad\text{for some}\quad j\in\left\{0,1,\ldots,N-1\right\}.
  \end{equation}
  Set $c_N:=2\cos x$. Evaluating \eqref{eq:q1_zero_41_polynomial_values} at $\vartheta=0$ and $\vartheta=x$, and using $\Phi_m\left(x\right)=0$ and $Nx=2\pi$, gives 
  $\mathcal H_m\left(2\right)=\mathcal K_N\left(2\right)=0$ and $\mathcal H_m\left(c_N\right)=\mathcal K_N\left(c_N\right)=0$. 
  Moreover, \eqref{eq:q1_zero_41_range} gives $N\geq6$, so $1\leq c_N<2$ and the two common zeros are distinct. We next prove the discrete negativity estimate
  \begin{equation}\label{eq:q1_zero_41_negative_interval}
    \Phi_m\left(\frac{2\pi k}{N}\right)<0\quad\text{for every integer}\quad 2\leq k\leq\left\lfloor\frac{N}{2}\right\rfloor.
  \end{equation}
  Fix an integer $k$ in this range. If $N$ is even and $k=N/2$, then $\Phi_m\left(\pi\right)=-2\left\lceil m/2\right\rceil<0$. Otherwise, set $s:=\pi k/N$, so $0<s<\pi/2$. 
  By \eqref{eq:q1_zero_41_range}, we have $n<N$ and $2N<3n$, and hence $ns\geq2\pi n/N>4\pi/3$. 
  If $\sin\left(ns\right)\leq0$, then $\sin\left(ns\right)/\sin s<n/4$. If $\sin\left(ns\right)>0$, then $ns>2\pi$, so $s>2\pi/n$. 
  Therefore $\sin s>\sin\left(2\pi/n\right)>4/n$, and again $\sin\left(ns\right)/\sin s<n/4$. 
  Since $2\leq k<N/2$, we also have $\min\left\{2s,\pi-2s\right\}\geq\pi/N$, and hence $\sin\left(2s\right)\geq\sin\left(\pi/N\right)\geq3/N>2/n$. 
  It follows that $\sin\left(2ns\right)/\sin\left(2s\right)>-n/2$. Applying \eqref{eq:q1_zero_41_identity} at $\vartheta=2s$ gives $4\Phi_m\left(2s\right)<n/2+n/2-n=0$, proving \eqref{eq:q1_zero_41_negative_interval}.
  For $2\leq j\leq N-2$, set $k_j:=\min\left\{j,N-j\right\}$. Then $2\leq k_j\leq\left\lfloor N/2\right\rfloor$. 
  Since $\Phi_m$ is even and $2\pi$-periodic, \eqref{eq:q1_zero_41_polynomial_values} and \eqref{eq:q1_zero_41_negative_interval} give 
  $\mathcal H_m\left(2\cos\left(2\pi j/N\right)\right)=4\Phi_m\left(2\pi k_j/N\right)<0$. 
  Combining this with \eqref{eq:q1_zero_41_K_roots}, the only distinct common zeros of $\mathcal H_m$ and $\mathcal K_N$ are $2$ and $c_N$.
  We now remove the known common zero $2$. Since $\mathcal K_N\left(2\right)=0$, the derivative at $2$ can be evaluated from its limit definition and then along $X=2\cos\left(\vartheta\right)$:
  \begin{equation*}
    \mathcal K_N'\left(2\right)=\lim_{X\to2}\frac{\mathcal K_N\left(X\right)}{X-2}
    =\lim_{\vartheta\to0}\frac{\mathcal K_N\left(2\cos\left(\vartheta\right)\right)}{2\cos\left(\vartheta\right)-2}
    =\lim_{\vartheta\to0}\left(\frac{\sin\left(N\vartheta/2\right)}{\sin\left(\vartheta/2\right)}\right)^2=N^2.
  \end{equation*}
  Hence $2$ is a simple zero of $\mathcal K_N$. Since both polynomials have integer coefficients and vanish at $2$, define
  \begin{equation*}
    \widetilde{\mathcal H}_m\left(X\right):=\frac{\mathcal H_m\left(X\right)}{X-2},\quad \widetilde{\mathcal K}_N\left(X\right):=\frac{\mathcal K_N\left(X\right)}{X-2}.
  \end{equation*}
  These reduced polynomials belong to $\mathbb Z[X]$. Since $c_N\neq2$, it remains a common zero, while $\widetilde{\mathcal K}_N\left(2\right)=\mathcal K_N'\left(2\right)=N^2\neq0$. 
  Thus $c_N$ is their only distinct common zero. 
  Let $\mathcal R_{N,m}\in\mathbb Q[X]$ be the monic greatest common divisor of $\widetilde{\mathcal H}_m$ and $\widetilde{\mathcal K}_N$ \cite[Sec. 8.1 and Sec. 9.2]{D-F}. 
  The polynomial $\mathcal R_{N,m}$ divides both reduced polynomials and can be expressed as a linear combination of them with coefficients in $\mathbb Q[X]$. 
  Hence its zeros are exactly their common zeros.
  Therefore $\mathcal R_{N,m}\left(X\right)=\left(X-c_N\right)^\nu$ for some positive integer $\nu$. 
  The coefficient of $X^{\nu-1}$ is $-\nu c_N$, so $\mathcal R_{N,m}\in\mathbb Q[X]$ implies $c_N\in\mathbb Q$.
  Choose $a\in\mathbb Z$ and $b\in\mathbb N$ such that $c_N=a/b$ and $\gcd\left(a,b\right)=1$. 
  Since $\mathcal K_N$ is monic with integer coefficients, multiplying $\mathcal K_N\left(a/b\right)=0$ by $b^N$ shows that $b$ divides $a^N$. 
  The condition $\gcd\left(a,b\right)=1$ forces $b=1$, and hence $c_N\in\mathbb Z$. Since $1\leq c_N<2$, we obtain $c_N=1$. 
  Thus $2\cos x=1$, and $0<x\leq\pi/3$ gives $x=\pi/3$ and $N=6$. This proves the claim. 
  
\end{proof}

We now prepare the proofs of Lemma \ref{lem:q2_positive_factors} and Lemma \ref{lem:q2_first_mode_sign}.
For the auxiliary calculation below and the proofs of these two lemmas, set $n:=2m+1$ and $y:=\pi/N$.
For every integer $k$, define
\begin{equation*}
  Q_k:=
  \begin{cases}
    n, & k\equiv0\pmod N,\\[0.4em]
    \displaystyle\frac{\sin\left(nky\right)}{\sin\left(ky\right)}, & k\not\equiv0\pmod N.
  \end{cases}
\end{equation*}
Since $n$ is odd, direct substitution gives
\begin{equation}\label{eq:q2_Qk_symmetry}
  Q_{k+N}=Q_k,\quad Q_{N-k}=Q_k,\quad k\in\mathbb Z.
\end{equation}
Thus $Q_k$ is $N$-periodic in $k$ and satisfies $Q_N=Q_0=n$.
We first establish the finite cosine-sum identity
\begin{equation}\label{eq:q2_finite_cosine_sum_Q}
  \sum_{r=1}^{m}\cos\left(\frac{2\pi kr}{N}\right)
  =\frac{Q_k-1}{2},\quad k\in\mathbb Z.
\end{equation}
If $k\equiv0\pmod N$, then both sides of \eqref{eq:q2_finite_cosine_sum_Q} equal
$m=\left(n-1\right)/2$.
If $k\not\equiv0\pmod N$, Lemma \ref{lem:finite_trig_sums}, applied with $x=2ky$, gives
\begin{equation*}
  \sum_{r=1}^{m}\cos\left(\frac{2\pi kr}{N}\right)
  =\frac{\sin\left(mky\right)\cos\left(\left(m+1\right)ky\right)}{\sin\left(ky\right)}
  =\frac{\sin\left(nky\right)-\sin\left(ky\right)}{2\sin\left(ky\right)}
  =\frac{Q_k-1}{2}.
\end{equation*}
Thus \eqref{eq:q2_finite_cosine_sum_Q} holds for every $k\in\mathbb Z$.
For $\ell\in[N-1]$, the definition of the mode factor and the product-to-sum identity give
\begin{equation*}
  \Gamma_{N,m,2}^{(\ell)}
  =\sum_{r=1}^{m}\cos\left(\frac{4\pi r}{N}\right)
  -\frac{1}{2}\sum_{r=1}^{m}\cos\left(\frac{2\pi\left(\ell-2\right)r}{N}\right)
  -\frac{1}{2}\sum_{r=1}^{m}\cos\left(\frac{2\pi\left(\ell+2\right)r}{N}\right).
\end{equation*}
Applying \eqref{eq:q2_finite_cosine_sum_Q} with $k=2$, $k=\ell-2$, and $k=\ell+2$, we obtain
\begin{equation}\label{eq:q2_Q_representation}
  \Gamma_{N,m,2}^{(\ell)}
  =\frac{2Q_2-Q_{\ell-2}-Q_{\ell+2}}{4},\quad \ell\in[N-1].
\end{equation}

\begin{proof}[\textbf{Proof of Lemma \ref{lem:q2_positive_factors}}]
  Set $\upsilon:=ny$, $c_\upsilon:=\cos\upsilon$, and $c_y:=\cos y$.
  Condition \eqref{eq:q2_positive_range} gives $0<\upsilon\leq\pi/3$.

  \medskip
  \noindent
  \textbf{Step 1. Claim that $\bm{2Q_2>Q_1+Q_3}$.}
  A direct simplification gives
  \begin{equation*}
    Q_2-\frac{Q_1+Q_3}{2}=-\frac{\sin\upsilon\left(c_\upsilon-c_y\right)\left(2c_\upsilon c_y-2c_y^2+1\right)}{\sin y\,c_y\left(4c_y^2-1\right)}.
  \end{equation*}
  Since $0<y<\upsilon\leq\pi/3$, we have $\sin\upsilon>0$, $c_\upsilon-c_y<0$, and $c_y>1/2$.
  Moreover, $c_\upsilon\geq1/2$, so $2c_\upsilon c_y-2c_y^2+1\geq1+c_y-2c_y^2=\left(1-c_y\right)\left(2c_y+1\right)>0$.
  The denominator in the displayed identity is positive, and hence $Q_2-\left(Q_1+Q_3\right)/2>0$.
  This proves the claim.

  \medskip
  \noindent
  \textbf{Step 2. Claim that $\bm{2Q_2>Q_0+Q_4}$.}
  Set $\psi:=2\upsilon$ and $\psi_0:=2y=\psi/n$.
  Then $0<\psi\leq2\pi/3$, and direct substitution gives
  \begin{equation*}
    2Q_2-Q_0-Q_4=\frac{\sin\psi}{\sin\psi_0}\left(2-\frac{\cos\psi}{\cos\psi_0}\right)-n.
  \end{equation*}

  \medskip
  \noindent
  \emph{Subcase 2.1: $0<\psi\leq\pi/2$.}
  Since $n\geq5$, we have $\psi_0\leq\psi/5<\psi/3$.
  Using $\sin\psi_0<\psi/n$ and $\cos\psi_0\geq\cos\left(\psi/3\right)$, together with $\cos\psi\geq0$, we obtain
  \begin{equation*}
    2Q_2-Q_0-Q_4
    >n\left[\frac{\sin\psi}{\psi}\left(2-\frac{\cos\psi}{\cos\left(\psi/3\right)}\right)-1\right]
    =n\left[\frac{\sin\psi}{\psi}\left(1+4\sin^2\left(\frac{\psi}{3}\right)\right)-1\right].
  \end{equation*}
  It remains to prove
  \begin{equation*}
    \sin\psi\left(1+4\sin^2\left(\psi/3\right)\right)>\psi.
  \end{equation*}
  Write $\psi=3w$ and define
  \begin{equation*}
    f(w):=\sin(3w)\left(1+4\sin^2w\right)-3w,\quad
    p(x):=80x^4+80x^3-56x^2-56x-3.
  \end{equation*}
  For $0<w\leq\pi/6$, direct differentiation gives
  \begin{equation*}
    f'(w)=\left(1-\cos w\right)p(\cos w).
  \end{equation*}
  Also, $p'(x)=8h(x)$, where
  \begin{equation*}
    h(x):=40x^3+30x^2-14x-7.
  \end{equation*}
  Since $h'(x)=120x^2+60x-14>0$ on $\left[\sqrt{3}/2,1\right]$ and $h\left(\sqrt{3}/2\right)=8\sqrt{3}+31/2>0$, both $h(x)$ and $p'(x)$ are positive on this interval.
  Thus $p$ is strictly increasing there.
  Since $p\left(\sqrt{3}/2\right)=2\sqrt{3}>0$, we have $f'\left(w\right)>0$ for $0<w\leq\pi/6$.
  Together with $f\left(0\right)=0$, this proves the desired inequality.

  \medskip
  \noindent
  \emph{Subcase 2.2: $\pi/2\leq\psi\leq2\pi/3$.}
  Since $\cos\psi\leq0$ and $0<\cos\psi_0\leq1$, the displayed identity gives
  \begin{equation*}
    2Q_2-Q_0-Q_4>n\left[\frac{\sin\psi}{\psi}\left(2-\cos\psi\right)-1\right].
  \end{equation*}
  It is enough to prove $\sin\psi\left(2-\cos\psi\right)>\psi$.
  Define $g\left(x\right):=\sin x\left(2-\cos x\right)-x$.
  On $\left[\pi/2,2\pi/3\right]$, we have $g'\left(x\right)=2\cos x\left(1-\cos x\right)\leq0$.
  Hence $g\left(\psi\right)\geq g\left(2\pi/3\right)=5\sqrt{3}/4-2\pi/3>0$.
  This proves the claim.

  \medskip
  \noindent
  \textbf{Step 3. Claim that $\bm{Q_5<Q_3}$.}
  The function $R_0\left(z\right):=\sin\left(nz\right)/\sin z$ is strictly decreasing on $\left(0,\pi/n\right)$.
  Indeed, $R_0'\left(z\right)/R_0\left(z\right)=n\cot\left(nz\right)-\cot z<0$, because $x\cot x$ is strictly decreasing on $\left(0,\pi\right)$.
  Since $m\geq2$ and $N\geq6m+3$, one has $N\geq15$, so $0<3y<5y<\pi$.
  If $5\upsilon<\pi$, then $3y<5y<\pi/n$, and the monotonicity of $R_0$ gives $Q_5<Q_3$.
  If $5\upsilon\geq\pi$, then $\pi\leq5\upsilon\leq5\pi/3<2\pi$, so $Q_5\leq0$, while $3\upsilon\leq\pi$ gives $Q_3\geq0$.
  If $Q_3=Q_5=0$, then $\sin\left(3\upsilon\right)=\sin\left(5\upsilon\right)=0$.
  Since $3$ and $5$ are coprime, this would force $\upsilon$ to be an integer multiple of $\pi$, contradicting $0<\upsilon\leq\pi/3$.
  Hence $Q_5<Q_3$.

  \medskip
  \noindent
  \textbf{Step 4. The tail comparison.}
  We prove $Q_k<Q_2$ for every $3\leq k\leq N-3$.
  By \eqref{eq:q2_Qk_symmetry}, it is enough to consider $3\leq k\leq N/2$.
  Set $z:=ky$.
  If $nz<\pi$, then $2y<z<\pi/n$, and the monotonicity of $R_0$ gives $Q_k<Q_2$.
  If $\sin\left(nz\right)\leq0$, then $Q_k\leq0<Q_2$.
  It remains only to treat the indices for which $nz\geq\pi$ and $\sin\left(nz\right)>0$.
  Since $\sin s>0$ for $s\in\left(2j\pi,\left(2j+1\right)\pi\right)$, $j\in\mathbb Z$, these indices satisfy $nz\geq2\pi$, and hence $z\geq2\pi/n$.
  Since $z\leq\pi/2$, we have $Q_k\leq1/\sin z\leq1/\sin\left(2\pi/n\right)$.
  On the other hand, since $R_0$ is strictly decreasing on $\left(0,\pi/n\right)$ and $2y\leq2\pi/\left(3n\right)$, we obtain $Q_2\geq\sin\left(2\pi/3\right)/\sin\left(2\pi/\left(3n\right)\right)$.
  Set $\beta:=2\pi/\left(3n\right)$.
  Since $n\geq5$, one has $0<\beta<\pi/6$ and $\sin\left(3\beta\right)/\sin\beta=3-4\sin^2\beta>2$.
  Therefore $\left(\sqrt{3}/2\right)\sin\left(3\beta\right)>\sin\beta$.
  Since $3\beta=2\pi/n$, this is equivalent to
  \begin{equation*}
    \frac{\sin\left(2\pi/3\right)}{\sin\left(2\pi/\left(3n\right)\right)}>\frac{1}{\sin\left(2\pi/n\right)}.
  \end{equation*}
  Combining the last inequality with the upper bound for $Q_k$ and the lower bound for $Q_2$, we obtain $Q_k<Q_2$.
  This proves the claim.

  \medskip
  \noindent
  \textbf{Step 5. Positivity of all mode factors.}
  We verify positivity through \eqref{eq:q2_Q_representation}.
  Since $m\geq2$ and $N\geq6m+3$, we have $N\geq15$, so the modes $1,2,3,4$ and their symmetric partners $N-1,N-2,N-3,N-4$ are distinct.
  For $\ell=1$ and $\ell=N-1$, Step 1 gives $2Q_2>Q_{\ell-2}+Q_{\ell+2}$.
  For $\ell=2$ and $\ell=N-2$, Step 2 gives $2Q_2>Q_{\ell-2}+Q_{\ell+2}$.
  For $\ell=3$ and $\ell=N-3$, Steps 1 and 3 give $Q_1+Q_5<Q_1+Q_3<2Q_2$.
  For $\ell=4$ and $\ell=N-4$, one of the indices $\ell-2$ and $\ell+2$ has circular distance $2$ from $0$, while the other has circular distance at least $3$. Hence one term equals $Q_2$ and the other is strictly smaller than $Q_2$.
  For every remaining mode, both indices have circular distance at least $3$ from $0$, so both terms are strictly smaller than $Q_2$.
  Thus $2Q_2>Q_{\ell-2}+Q_{\ell+2}$ for every $\ell\in[N-1]$, and \eqref{eq:q2_Q_representation} gives $\Gamma_{N,m,2}^{(\ell)}>0$ for every $\ell\in[N-1]$.
\end{proof}

\begin{proof}[\textbf{Proof of Lemma \ref{lem:q2_first_mode_sign}}]
  Retaining the notation above, set $\upsilon:=ny$, $c_\upsilon:=\cos\upsilon$, and $c_y:=\cos y$.
  We first reduce the sign of $\Gamma_{N,m,2}^{(1)}$ to that of a scalar expression.
  By \eqref{eq:q2_Q_representation}, $2\Gamma_{N,m,2}^{(1)}=Q_2-\left(Q_1+Q_3\right)/2$.
  A direct simplification gives
  \begin{equation*}
    2\Gamma_{N,m,2}^{(1)}=-\frac{\sin\upsilon\left(c_\upsilon-c_y\right)\left(2c_\upsilon c_y-2c_y^2+1\right)}{\sin y\,c_y\left(4c_y^2-1\right)}.
  \end{equation*}
  Since $N\geq2m+2$, we have $0<y<\upsilon<\pi$.
  Also, $N\geq4$ gives $c_y>0$ and $4c_y^2-1>0$.
  Since $c_\upsilon-c_y<0$, the sign of $\Gamma_{N,m,2}^{(1)}$ is the sign of $\chi:=2c_\upsilon c_y-2c_y^2+1$.
  Define $r_\ast:=c_y-1/\left(2c_y\right)$ and $\Delta:=3n-N$.
  Then $\chi\leq0$ is equivalent to $c_\upsilon\leq r_\ast$.
  Condition \eqref{eq:q2_lower_intermediate_range} gives $\Delta\geq1$, and the identity $\upsilon=\left(\pi+\Delta y\right)/3$ follows from $\Delta=3n-N$.
  We shall use below that the function $x\mapsto4x^3-3x$ is strictly decreasing on $\left[0,1/2\right]$.

  \medskip
  \noindent
  \emph{Case 1. $\Delta\geq2$.}
  Suppose first that $\Delta\geq2$, and set $d_2:=\cos\left(\left(\pi+2y\right)/3\right)$.
  Since the cosine function is decreasing on $\left[0,\pi\right]$, the identity $\upsilon=\left(\pi+\Delta y\right)/3$ gives $c_\upsilon\leq d_2$.
  Moreover, $N\geq4$ implies $c_y\in\left[\sqrt{2}/2,1\right)$ and $d_2,r_\ast\in\left[0,1/2\right]$.
  The triple-angle identity gives $4d_2^3-3d_2=1-2c_y^2$, while direct simplification yields
  \begin{equation*}
    4r_\ast^3-3r_\ast-\left(1-2c_y^2\right)=\frac{\left(c_y-1\right)\left(2c_y^2-1\right)\left(4c_y^3+6c_y^2-c_y-1\right)}{2c_y^3}\leq0.
  \end{equation*}
  The cubic factor $4c_y^3+6c_y^2-c_y-1$ is positive on $\left[\sqrt{2}/2,1\right)$, because its derivative is positive there and its value at $\sqrt{2}/2$ is $2+\sqrt{2}/2>0$.
  Since the function $x\mapsto4x^3-3x$ is strictly decreasing on $\left[0,1/2\right]$, we obtain $r_\ast\geq d_2\geq c_\upsilon$.
  Equality $c_\upsilon=r_\ast$ would require equality in both inequalities $r_\ast\geq d_2$ and $d_2\geq c_\upsilon$, hence $\Delta=2$ and $c_y=\sqrt{2}/2$.
  The latter condition gives $N=4$, whereas $\Delta=3n-N=2$ would then give $n=2$, contradicting $n=2m+1$.
  Therefore $c_\upsilon<r_\ast$, and hence $\Gamma_{N,m,2}^{(1)}<0$.

  \medskip
  \noindent
  \emph{Case 2. $\Delta=1$.}
  It remains to consider $\Delta=1$.
  Then $N=3n-1$.
  Set $d_1:=\cos\left(\left(\pi+y\right)/3\right)$.
  Since $\upsilon=\left(\pi+y\right)/3$, we have $c_\upsilon=d_1$.
  Since $m\geq1$, we have $n=2m+1\geq3$, and therefore $N=3n-1\geq8$.
  Hence $c_y\geq\cos\left(\pi/8\right)$.
  Moreover, $d_1,r_\ast\in\left[0,1/2\right]$, and the triple-angle identity gives $4d_1^3-3d_1=-c_y$.
  A direct simplification gives
  \begin{equation*}
    4r_\ast^3-3r_\ast+c_y=\frac{\left(c_y-1\right)\left(c_y+1\right)\left(8c_y^4-8c_y^2+1\right)}{2c_y^3}\leq0.
  \end{equation*}
  Indeed, $8c_y^4-8c_y^2+1=8\left(c_y^2-\cos^2\left(\pi/8\right)\right)\left(c_y^2-\sin^2\left(\pi/8\right)\right)\geq0$.
  Equality in the last displayed inequality holds only when $c_y=\cos\left(\pi/8\right)$, which is equivalent to $N=8$.
  Since $N=3n-1$, this gives $n=3$ and therefore $m=1$.
  Using again that the function $x\mapsto4x^3-3x$ is strictly decreasing on $\left[0,1/2\right]$, we obtain $c_\upsilon=d_1\leq r_\ast$, with equality exactly at $\left(N,m\right)=\left(8,1\right)$.
  Hence $\Gamma_{N,m,2}^{(1)}\leq0$, and equality occurs exactly at $\left(N,m\right)=\left(8,1\right)$.
  Combining the two cases proves the result.
\end{proof}

\begin{proof}[\textbf{Proof of Lemma \ref{lem:q2_degenerate_instability}}]
  For the eight-oscillator nearest-neighbor model, the edgewise phase difference of the $2$-twisted profile is $\pi/2$ on every associated edge.
  Hence every linearized edge weight contains the factor $\cos\left(\pi/2\right)=0$, and $\Gamma_{8,1,2}^{(\ell)}=0$ for all $\ell\in[7]$.
  Thus the DFT-based criterion gives no conclusion, and we use a nonlinear reduction.

  \medskip
  \noindent
  \textbf{Step 1. The four-periodic invariant subsystem.}
  Let $F_4$ and $F_8$ denote the vector fields of the four-oscillator and eight-oscillator nearest-neighbor models with the same coupling strength $K>0$.
  Define the linear duplication map $\widehat{\iota}:\mathbb R^4\longrightarrow\mathbb R^8$ by
  \begin{equation*}
    \widehat{\iota}\left(u_1,u_2,u_3,u_4\right):=\left(u_1,u_2,u_3,u_4,u_1,u_2,u_3,u_4\right).
  \end{equation*}
  Let $\iota:\mathbb T^4\longrightarrow\mathbb T^8$ be the map induced by $\widehat{\iota}$, and set
  \begin{equation*}
    \mathcal S:=\iota\left(\mathbb T^4\right)=\left\{\theta\in\mathbb T^8:\theta_{i+4}=\theta_i,\quad i\in[4]\right\}.
  \end{equation*}
  For every $x\in\mathbb T^4$ and $i\in[4]$,
  \begin{equation*}
    \left(F_8\left(\iota\left(x\right)\right)\right)_i=K\sin\left(x_{\langle i+1\rangle_4}-x_i\right)=\left(F_4\left(x\right)\right)_i,
  \end{equation*}
  and the component with index $i+4$ has the same value.
  Therefore, under the standard identifications of tangent spaces with Euclidean spaces,
  \begin{equation*}
    F_8\left(\iota\left(x\right)\right)=\widehat{\iota}\left(F_4\left(x\right)\right),\quad x\in\mathbb T^4.
  \end{equation*}
  Hence $\mathcal S$ is invariant under the eight-oscillator flow.

  \medskip
  \noindent
  \textbf{Step 2. Identification of the twisted profiles and quotient diameter.}
  Let $\phi_8^{(2)}$ and $\phi_4^{(1)}$ denote the profiles in \eqref{eq:q_twisted_profile_definition} with $\left(N,q\right)=\left(8,2\right)$ and $\left(N,q\right)=\left(4,1\right)$, respectively.
  Set $\zeta:=2\pi\left(0,0,0,0,1,1,1,1\right)$.
  For the chosen real-valued representatives, we have
  \begin{equation*}
    \phi_8^{(2)}=\widehat{\iota}\left(\phi_4^{(1)}\right)+\zeta.
  \end{equation*}
  Both profiles have common angular velocity $K$, since their nonzero associated edgewise phase differences are equal to $\pi/2$ modulo $2\pi$.
  For $d\in\left\{4,8\right\}$ and $u\in\mathbb R^d$, write
  \begin{equation*}
    D_d\left(u\right):=\max_{i\in[d]}u_i-\min_{i\in[d]}u_i.
  \end{equation*}
  Duplication preserves the quotient diameter, namely $D_8\left(\widehat{\iota}\left(u\right)\right)=D_4\left(u\right)$ for every $u\in\mathbb R^4$.

  \medskip
  \noindent
  \textbf{Step 3. A family of four-periodic perturbations.}

  \noindent
  For $\delta>0$, set $v^\delta:=\left(0,-\delta/6,-\delta/3,-\delta/2\right)$.
  Define $x^\delta\left(0\right):=\phi_4^{(1)}+v^\delta$, that is,
  \begin{equation*}
    x^\delta\left(0\right)=\left(0,\frac{\pi}{2}-\frac{\delta}{6},\pi-\frac{\delta}{3},\frac{3\pi}{2}-\frac{\delta}{2}\right).
  \end{equation*}
  Its cyclic phase-difference representatives in $[0,2\pi)$ are
  \begin{equation*}
    P_0\left(\delta\right):=\left(\frac{\pi}{2}-\frac{\delta}{6},\frac{\pi}{2}-\frac{\delta}{6},\frac{\pi}{2}-\frac{\delta}{6},\frac{\pi}{2}+\frac{\delta}{2}\right).
  \end{equation*}
  It was proved in \cite[Sec. 4.1]{H-K} that, for every sufficiently small $\delta>0$, the corresponding four-oscillator trajectory enters the basin of the synchronized orbit and converges to it.
  The normalization of the coupling strength in that reference is immaterial here, because changing $K>0$ only amounts to a positive rescaling of time.
  Let $x^\delta\left(t\right)$ denote the solution of the lifted four-oscillator system with initial value $x^\delta\left(0\right)$, and define
  \begin{equation*}
    \widetilde{\theta}^{\delta}\left(t\right):=\widehat{\iota}\left(x^\delta\left(t\right)\right)+\zeta.
  \end{equation*}
  By the duplication identity in Step 1 and the $2\pi$-periodicity of the vector field, $\widetilde{\theta}^{\delta}\left(t\right)$ is a solution of the lifted eight-oscillator system.
  Moreover,
  \begin{equation*}
    \widetilde{\theta}^{\delta}\left(0\right)=\phi_8^{(2)}+\widehat{\iota}\left(v^\delta\right).
  \end{equation*}
  Its projection to $\mathbb T^8$ is the image under $\iota$ of the projection of $x^\delta\left(0\right)$ to $\mathbb T^4$, and the initial quotient diameter is
  \begin{equation}\label{eq:q2_initial_diameter}
    D_8\left(\widetilde{\theta}^{\delta}\left(0\right)-\phi_8^{(2)}\right)=\frac{\delta}{2}.
  \end{equation}
  Thus these initial data approach the $2$-twisted dancing orbit in quotient diameter as $\delta\to0$.
  By construction, the projection of $\widetilde{\theta}^{\delta}\left(t\right)$ to $\mathbb T^8$ is the image under $\iota$ of the projection of $x^\delta\left(t\right)$ to $\mathbb T^4$.
  Hence this projected eight-oscillator trajectory converges to the synchronized orbit.

  \medskip
  \noindent
  \textbf{Step 4. Escape in the quotient.}
  The synchronized orbit and the $2$-twisted dancing orbit represent distinct points in the quotient phase space.
  Choose a quotient neighborhood $\mathcal V$ of the point represented by the $2$-twisted dancing orbit whose closure does not contain the point represented by the synchronized orbit.
  Then choose $\eta>0$ such that the quotient-diameter ball of radius $\eta$ about the $2$-twisted dancing orbit is contained in $\mathcal V$.
  Given any sufficiently small $\varepsilon>0$, choose $\delta>0$ so small that the convergence result in \cite[Sec. 4.1]{H-K} applies and $0<\delta/2<\min\left\{\varepsilon,\eta\right\}$.
  By \eqref{eq:q2_initial_diameter}, the corresponding initial point satisfies
  \begin{equation*}
    0<D_8\left(\widetilde{\theta}^{\delta}\left(0\right)-\phi_8^{(2)}\right)<\varepsilon.
  \end{equation*}
  As long as the quotient trajectory remains in $\mathcal V$, the lifted solution $\widetilde{\theta}^{\delta}\left(t\right)$ gives the corresponding continuous representative for comparison with $\phi_8^{(2)}+Kt\mathbf 1$.
  Since the quotient trajectory converges to the point represented by the synchronized orbit, it eventually leaves $\mathcal V$, and hence it also leaves the quotient-diameter ball of radius $\eta$.
  The function
  \begin{equation*}
    t\longmapsto D_8\left(\widetilde{\theta}^{\delta}\left(t\right)-Kt\mathbf 1-\phi_8^{(2)}\right)
  \end{equation*}
  is continuous and has initial value strictly smaller than $\eta$.
  Therefore, there exists a first time $t_\delta>0$ such that
  \begin{equation*}
    D_8\left(\widetilde{\theta}^{\delta}\left(t_\delta\right)-Kt_\delta\mathbf 1-\phi_8^{(2)}\right)=\eta.
  \end{equation*}
  Since $\eta$ is independent of $\varepsilon$ and $\delta$, the $2$-twisted dancing orbit is nonlinearly unstable in the quotient sense.
\end{proof}

\end{document}